\documentclass[a4paper,11pt]{article}

\usepackage[english]{babel}
\usepackage[english]{translator}
\usepackage[T1]{fontenc}
\usepackage[utf8]{inputenc}

\usepackage{amsmath}
\usepackage{amssymb}
\usepackage{amsfonts}
\usepackage{amstext}
\usepackage{amsthm}
\usepackage{bbm}

\usepackage{graphicx}
\usepackage{color}
\usepackage{subfigure}
\usepackage{float}

\usepackage{marvosym}
\usepackage{enumerate}
\usepackage{enumitem}
\usepackage{todonotes}
\usepackage{hyperref}

\usepackage[
  hmarginratio={1:1},     
  vmarginratio={1:1},     
  textwidth=460pt,        
  heightrounded,          
  bottom=3.65cm,
  top=3.65cm,
]{geometry}

\hypersetup{
  pdffitwindow=false,
  pdfhighlight=/O,
  pdfnewwindow,
  colorlinks=true,
  citecolor=red,             
  linkcolor=blue,            
  menucolor=blue,            
  urlcolor=blue,             
  pdfpagemode=UseOutlines,
  bookmarksnumbered=true,
  linktocpage,
  pdfkeywords={},
  pdfcreator={pdflatex},
  pdfproducer={LaTeX with hyperref}
}

\newcommand{\D}{\mathcal{D}}
\newcommand{\K}{\mathbb{K}}
\newcommand{\C}{\mathbb{C}}
\newcommand{\R}{\mathbb{R}}

\newcommand{\N}{\mathbb{N}}

\newcommand{\G}{\mathcal{G}}
\newcommand{\cO}{\mathcal{O}}
\renewcommand{\S}{\mathcal{S}}
\renewcommand{\Re}{\mathrm{Re}\,}
\renewcommand{\Im}{\mathrm{Im}\,}

\newcommand{\A}{\mathcal{A}}

\newcommand{\M}{\mathcal{M}}

\newcommand{\one}{\mathbbm{1}}
\newcommand{\U}{\mathcal{U}}
\newcommand{\disp}{\mathrm{disp}}
\newcommand{\clos}{\mathcal{C}}
\newcommand{\Cexp}{C_{\mathrm{exp}}}
\newcommand{\Calg}{C_{\mathrm{alg}}}
\renewcommand{\ker}{\mathcal{N}}
\newcommand{\ran}{\mathcal{R}}
\renewcommand{\dim}{\mathrm{dim}\,}

\newcommand{\st}{\mathfrak{s}}
\newcommand{\un}{\mathfrak{u}}
\renewcommand{\c}{\mathfrak{c}}

\newcommand{\pt}{\mathrm{pt}}
\newcommand{\ess}{\mathrm{ess}}
\renewcommand{\P}{\mathcal{P}}

\newcommand{\sect}
{
  \setcounter{equation}{0}
  \setcounter{figure}{0}
  \section
}

\theoremstyle{definition}
\newtheorem{definition}{Definition}[section]
\newtheorem{assumption}{Assumption}

\newtheorem{remark}[definition]{Remark}

\theoremstyle{plain}
\newtheorem{theorem}[definition]{Theorem}
\newtheorem{lemma}[definition]{Lemma}
\newtheorem{corollary}[definition]{Corollary}
\newtheorem{proposition}[definition]{Proposition}

\allowdisplaybreaks

\begin{document}
\title{Algebraic rates of stability
for front-type modulated waves \\ in Ginzburg Landau equations}
\setlength{\parindent}{0pt}

\begin{center}
{\Large Algebraic rates of stability
for front-type modulated waves \\ in Ginzburg Landau equations} \\
\vspace{12pt}
Wolf-J\"urgen Beyn\footnotemark[1] and Christian D\"oding\footnotemark[2] \\
\vspace{12pt}
April 12, 2024
\end{center}

\footnotetext[1]{Department of Mathematics, Bielefeld University, 33501 Bielefeld, Germany, \\ e-mail: \textcolor{blue}{beyn@math.uni-bielefeld.de}.}
\footnotetext[2]{Institute for Numerical Simulation, University of Bonn, 53115 Bonn, Germany, \\ e-mail: \textcolor{blue}{doeding@ins.uni-bonn.de}.}

\noindent
\begin{center}
\begin{minipage}{0.8\textwidth}
  {\small
    \textbf{Abstract.}
  We consider the stability of front-type modulated waves in the complex Ginzburg-Landau equation (CGL). The waves occur
  in the bistable regime (e.g. of the quintic CGL) and connect the zero state to a spatially homogenous state oscillating in time.
  For initial perturbations that decay at a certain algebraic rate, we prove convergence to the wave with asymptotic
  phase. The convergence holds in algebraically weighted Sobolev norms and with an algebraic rate in time, where
  the asymptotic phase is approached by one order less than the profile. On the technical side we use the theory of exponential
  trichotomies to separate the spatial modes into growing, weakly decaying, and strongly decaying ones. This allows us
  to derive resolvent and semigroup estimates in weighted Sobolev norms and to close the argument with a Gronwall
  lemma involving algebraic weights.
}
\end{minipage}
\end{center}

\vspace{12pt}

\vspace{12pt}
 
\sect{Introduction}

The topic of this paper is the stability of  front-type modulated waves in complex-valued semilinear parabolic evolution equations of the form 
\begin{align} \label{Evo}
	U_t = \alpha U_{xx} + G(|U|^2)U, \quad x \in \R,\, t \ge 0,
\end{align}
where $U = U(x,t) \in \C$, $\alpha \in \C$ with $\Re \alpha > 0$, and a smooth nonlinearity $G: \R_+ \rightarrow \C$
is given. A prototypical example of such an equation is the quintic Ginzburg-Landau equation (QCGL)
for which $G$ is a quadratic polynomial
\begin{align} \label{QCGL}
  G(|U|^2)=  \beta_0 + \beta_2 |U|^2 + \beta_4 |U|^4, \quad \beta_0,\beta_2,\beta_4 \in \C
\end{align}
and which we use for illustration.
We consider modulated waves of the form
\begin{align}\label{MW}
	U_\star(x,t) = e^{- i \omega t} V_\star(x - ct) 
\end{align}
which have frequency $\omega \in \R$, velocity $c >0$, and a profile $V_\star: \R \rightarrow \C$ with a front-like shape, i.e.,
\begin{align} \label{Vasym}
	V_\star(x) \rightarrow \begin{cases} V_\infty, & \text{ as } x \rightarrow + \infty, \\
	0, & \text{ as } x \rightarrow - \infty \end{cases}
\end{align}
for some $V_\infty \in \C$, $V_\infty \neq 0$. In Figure \ref{Figure:TOF3D} we illustrate the profile of such a wave and its motion
in an $(x,\Re V, \Im V)$ frame. The wave belongs to the class of front-type modulated waves as classified
by  Sandstede and Scheel in \cite[Sec. 2.2]{SandstedeScheel04}. However, it does not asymptote to wave trains in both
directions $x \to \pm \infty$, hence, in a strict sense, does  not fall into their category of defects;
see \cite[Introduction]{BeynDoeding22} for a more detailed comparison. Note that we used   the
term traveling oscillating front (TOF) in  \cite{BeynDoeding22} to denote this specific type of wave.
Waves of the form \eqref{MW} with \eqref{Vasym} occur in the so called bistable regime, i.e. they connect the stable equilibrium
$0$ at $- \infty$ with the spatially homogeneous state $V_{\infty}$ at $+\infty$ which belongs to a stable limit
cycle of the ODE $U_t=G(|U|^2)U$ obtained from \eqref{Evo} as  $x \to +\infty$.
Numerical simulations of such waves are shown in \cite[Introduction]{BeynDoeding22}, and we note
that there is numerical and theoretical evidence for their generic occurrence in QCGL (cf. \cite{Saarloos}).

\begin{figure}[h!]
\centering
\includegraphics[scale=0.4]{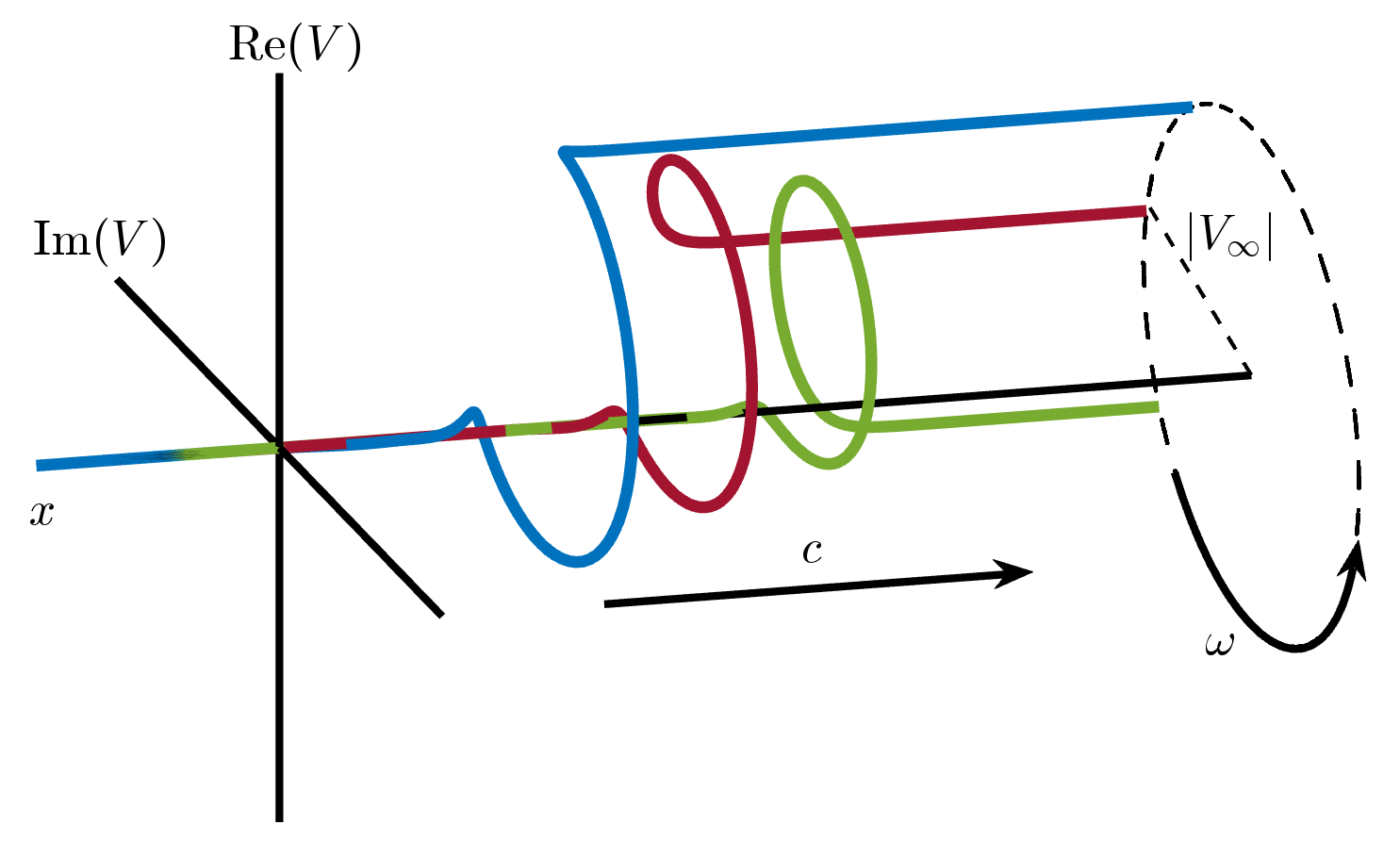}
\caption{Front-like modulated wave at three successive time instances (blue, red, green).} \label{Figure:TOF3D}
\end{figure}

The goal of this work is to prove asymptotic stability with asymptotic phase of the solutions to the perturbed problem
\begin{equation*}
  U_t= \alpha U_{xx}+G(|U|^2)U, \quad U(\cdot,0)= V_{\star} + V_0.
\end{equation*}
Our main result Theorem \ref{mainresult} assumes the initial perturbation $V_0$ and its derivative to decay algebraically sufficiently fast, i.e.
\begin{align} \label{eq0:asyminit}
  |V_0(x)|, |\partial_x V_0(x)| & \asymp (1+x^2)^{-\frac{k+m}{2}}, \quad k > m > \tfrac{9}{2}.
\end{align}
Then we conclude algebraic decay both in space and time for the perturbed solution as follows:
\begin{equation} \label{eq0:asymsol}
\begin{aligned}
  U(x,t)-V_{\star}(x- \tau(t))  & \asymp (1+x^2)^{-\frac{k}{2}}(1+t)^{1-\frac{m^{\star}}{2}}, \quad  \tau(t)-\tau_{\infty} \asymp (1+t)^{2-\frac{m^{\star}}{2}}.
\end{aligned}
\end{equation}
where $m^{\star}= \lfloor m \rfloor  + \max(0,2q-1)$, $q=m-\lfloor m \rfloor$.
The value $\tau_{\infty}$ is the asymptotic phase which is approached by one order less in time than the shifted
profile. In fact, the asymptotics above hold in suitable weighted Sobolev spaces as specified in
\eqref{eq2:Lspace}, \eqref{eq2:Wspace}.
The result complements our main theorem from \cite{BeynDoeding22} where $V_0$ was allowed to
be the sum of an exponentially localized  function and of a profile which limits at $+ \infty$ to a small but nonzero
value.

The coupling of algebraic decay rates in space for initial perturbations  to algebraic decay rates in time is a well-known phenomenon
in stability studies of nonlinear waves. As a general rule, it occurs when the essential spectrum of the linearization touches
the imaginary axis at the origin. In the following we discuss some results of this type in the literature and
expand on differences and similarities to our work.

One of the most studied topics is the so-called marginal stability of critical fronts
in the Ginzburg-Landau equation; see the work of Avery and Scheel \cite{AveryScheel21,AveryScheel22,AveryScheel24}
and the early contribution \cite{Bricmont94}.
The recent works \cite{AveryScheel22,AveryScheel24} give a rather comprehensive overview of
the development of the theory. Critical fronts typically occur in the cubic Ginzburg-Landau equation and belong to the
monostable case. There are several differences when compared to the bistable case above:
in addition to translation and gauge invariance which is also present in our case, there is a whole
interval of velocities of which only the extreme one is marginally stable; as a result one has
to require  exponential decay of the invading perturbation which is not needed in our result above.
The proofs have several similarities, such as the algebraically weighted resolvent estimates via suitable contours,
but differ in technical details, such as the additive decomposition used in \eqref{eq0:asymsol} when  compared
to the polar coordinate ansatz used in  \cite{AveryScheel24}.
We also mention that the time rates for the critical fronts in \cite{AveryScheel24} are shown to be optimal
while  the threshold $\tfrac{9}{2}$ in \eqref{eq0:asyminit}, \eqref{eq0:asymsol}  is
open for improvement. However, we expect, as a general phenomenon, that the spatial rate $k+m$ for the initial data in \eqref{eq0:asyminit} splits
into a spatial rate $k$ and a temporal rate $m$ for the perturbed solution in \eqref{eq0:asymsol}.

Next we  mention the work of Kapitula \cite{Kapitula91,Kapitula94} who derives a stability result
w.r.t. the $L^{\infty}$-norm for a version of \eqref{Evo} with a specific form of the nonlinearity
and a small imaginary part of the diffusion coefficient $\alpha$; see \cite[Theorem 7.2]{Kapitula94}. The analysis proceeds
via the associated Evans function
and algebraic time estimates are restricted to a particular parameter set ($\rho=0$ in  \cite[Theorem 7.2]{Kapitula94}),
while exponential weights are required for the general case. Moreover, the decomposition of the dynamics is
slightly different from ours and some technical arguments
  of the algebraic estimates need to be corrected; see Lemma \ref{est} and Remark \ref{rem:Kapitula} for more details.

  In the following we give a rough outline of the subsequent sections and emphasize the major mathematical
  arguments.
  In Section \ref{sec1} we list the assumptions
  and formulate our main result Theorem \ref{mainresult} in algebraically weighted Sobolev spaces.
  In Section \ref{sec3} we analyze in detail the spectrum of the operator obtained by linearizing about the
  wave \eqref{MW} in a co-moving frame. Our assumptions lead to quadratic contact of the essential spectrum with the imaginary
  axis (Figure \ref{essspec}) and we use the theory of exponential trichotomies (Appendix  \ref{appendixA}) to
  handle the separation of spatial modes on the positive axis into a weakly and a strongly exponentially
  decaying part. This serves as a suitable preparation for the resolvent estimates (Proposition \ref{propEV}) which
  hold in a crescent-like domain of the complex plane (Figure \ref{essspec}). 
  In the ensuing sections we derive algebraic decay rates for the semigroup generated by the linearized operator
  (Theorem \ref{poly4:semigroup}) and decompose the nonlinear dynamics into an equation along the shifted
  profile and orthogonal to it (Lemma \ref{poly:lemmatrafo}, Eq.\eqref{wtau-integral}). Using remainder estimates
  in algebraically weighted Sobolev norms (Lemma \ref{poly:estimates}) and a Gronwall lemma adapted
  to algebraic integral kernels (Lemma \ref{poly:Gronwall}) allows us to conclude nonlinear stability
  (Theorem \ref{poly:Stabilitywt}).
  
\sect{Assumptions and main results}
\label{sec1}
Instead of working with the complex-valued equation \eqref{Evo} it is convenient to analyze the corresponding real-valued system as a classical reaction diffusion equation. With $U = u_1 + i u_2$, $u_j(x,t) \in \R$, $\alpha = \alpha_1 + i \alpha_2$, $\alpha_j \in \R$ and $G = g_1 + i g_2$, $g_j: \R \rightarrow \R$ we rewrite \eqref{Evo}  equivalently  as
\begin{align} \label{rEvo}
	u_t = A u_{xx} + f(u), \quad x \in \R, \, t \ge 0,
\end{align}
where
\begin{align} \label{SystemDef}
	A = \begin{pmatrix}
	\alpha_1 & - \alpha_2 \\ \alpha_2 & \alpha_1
	\end{pmatrix}, \quad f(u) = g(|u|^2)u, \quad g(\cdot) = \begin{pmatrix}
	g_1(\cdot) & -g_2(\cdot) \\ g_2(\cdot) & g_1(\cdot)
	\end{pmatrix}.
\end{align}
A front-type modulated wave of \eqref{rEvo} is then given by a particular solution of the form
\begin{align*}
	u_\star(x,t) = R_{-\omega t} v_\star(x-ct), \quad R_\theta = \begin{pmatrix}
	\cos \theta & -\sin \theta \\ \sin \theta & \cos \theta
	\end{pmatrix}
\end{align*}
where $c >0$ and the profile $v_\star$ satisfies
\begin{equation} \label{eq2:vasymp}
	\lim_{\xi \rightarrow -\infty} v_\star(\xi) = 0, \quad \lim_{\xi \rightarrow +\infty} v_\star(\xi) = v_\infty \in \R^2.
\end{equation}
Here the asymptotic rest state $v_\infty \in \R^2$ is given by $v_\infty = (\Re V_\infty, \Im V_\infty)^\top$. Next, it is natural to transform \eqref{rEvo} into a co-moving frame by setting $u(x,t) = R_{-\omega t} v(\xi,t)$, $\xi = x-ct$ so that $u$ solves \eqref{rEvo}
if and only if $v$ solves 
\begin{align} \label{comovsys}
	v_t = A v_{\xi\xi} + cv_\xi + S_\omega v + f(v), \quad x \in \R, \, t \ge 0, \, S_\omega = \begin{pmatrix}
	0 & -\omega \\ \omega & 0
	\end{pmatrix}.
\end{align}
Then the profile $v_\star$ becomes a stationary solution of \eqref{comovsys}, i.e.,
\begin{align} \label{statcomov}
	0 = A v_{\star,xx} + cv_{\star,x} + (S_\omega + g(v_\star)) v_\star, \quad x \in \R.
\end{align}
To obtain the stability of the stationary solution $v_\star$ of \eqref{comovsys}, we
study the long time dynamics of the perturbed initial value problem
\begin{align} \label{perturbsys}
	v_t = A v_{xx} + cv_x + S_\omega v + f(v), \quad v(0) = v_\star + u_0 
\end{align}
where the initial perturbation $u_0$ is small in a suitable sense. The stability behavior or long time dynamics of the solution of \eqref{perturbsys} strongly depends on the norm in which  $u_0$ is measured. Before making this precise, let us collect our basic assumptions for the system \eqref{perturbsys}.
\begin{assumption} \label{A1}
The coefficient $\alpha$ and the function $g$ satisfy
\begin{align*}
	\alpha_1 > 0, \quad g \in C^3(\R,\R^{2,2}), \quad g_1(0) < 0.
\end{align*}
\end{assumption}

\begin{assumption} \label{A2}
There exists a front-type modulated wave solution $u_\star$ of \eqref{rEvo} with profile $v_\star \in C^2_b(\R,\R^2)$, speed $c > 0$, frequency $\omega \in \R$, $\omega \neq0$ and asymptotic rest state $v_\infty= ( |v_\infty|, 0 )^\top \in \R^2$ such that
\begin{align*}
	g_1'(|v_\infty|^2) < 0.
\end{align*}  
\end{assumption}
Recall that one can multiply the wave $U_{\star}$ by a phase factor $e^{i \theta}$ so that the rest state has this particular form.
\begin{assumption} \label{A3}
The coefficient $\alpha$ and the function $g$ satisfy
\begin{align*}
	\alpha_1 g_1'(|v_\infty|^2) + \alpha_2 g_2'(|v_\infty|^2) < 0. 
\end{align*}
\end{assumption}

Recall from \cite[Lemma 2.4]{BeynDoeding22} that \eqref{eq2:vasymp} and Assumption \ref{A1} imply
\begin{equation*}
  g(|v_{\infty}|^2)=-S_{\omega}, \quad \lim_{x \to \pm \infty}v_{\star,x}(x)=\lim_{x \to \pm \infty}v_{\star,xx}(x)=0.
\end{equation*}

The linearization $L$ of \eqref{perturbsys} at the profile $v_\star$ is given by the operator
\begin{align} \label{LatL}
	Lu = Au_{xx} + cu_x + S_\omega u + Df(v_\star) u.
\end{align}
In general, we consider $L$ as a linear operator on a Banach space $X$ defined on a domain $\D(L) \subset X$ and write $L \in \clos[X]$ if $L$ is  a closed and densely defined linear operator.  A simple example is  $X = L^2$ so that $L \in \clos[L^2]$ with $\D(L) = H^2$. In any case, for $L \in \clos[X]$ we denote its resolvent set by
\begin{align*}
	\mathrm{res}(L) := \{ s \in \C: \, sI-L: \D(L) \rightarrow X \text{ is bijective} \}
\end{align*}
and its spectrum by $\sigma(L) = \C \setminus \mathrm{res}(L)$. The spectrum is subdivided into the point spectrum 
\begin{align*}
	\sigma_{\pt}(L) := \{ s \in \sigma(L):\, sI - L \text{ is Fredholm of index } 0 \}
\end{align*}
and the essential spectrum $\sigma_{\ess}(L) = \sigma(L) \setminus \sigma_{\pt}(L)$. Let us denote by $\ker(L) \subset \D(L)$ the nullspace of $L \in \clos[X]$.

\begin{assumption} \label{A4}
For the operator $L \in \clos[L^2]$ from \eqref{LatL} with $\D(L) = H^2$ there is $\beta_E > 0$ such that $\Re s < - \beta_E$ holds for all $s \in \sigma_{\mathrm{\pt}}(L)$. Moreover,
\begin{align*}
	\dim \ker(L) = \dim \ker(L^2) = 1.
\end{align*}
\end{assumption}
Indeed, we show that $\ker(L)$ is spanned by $v_{\star,x}$. However, it turns out that $L \in \clos[L^2]$ is not Fredholm of index $0$ so that $0 \in \sigma_{\ess}(L)$ and $v_{\star,x}$ does not belong to an eigenvalue in $\sigma_{\mathrm{pt}}(L)$. \\
In order to specify the  smallness of the perturbation $u_0$ in \eqref{perturbsys}, we introduce the algebraic weight function of linear growth
\begin{align*}
	\eta(x) := (x^2 + 1)^{\frac{1}{2}}, \quad x \in \R.
\end{align*}
Further, let us define associated algebraically weighted Lebesgue spaces for arbitrary $k \ge0$:
\begin{equation} \label{eq2:Lspace}
  \begin{aligned}
  L^2_k(\R,\R^n) := \{ v \in L^2(\R,\R^n):\, \eta^k v \in L^2(\R,\R^n) \}, \quad (v,w)_{L^2_k}:= (\eta^k v, \eta^k w)_{L^2},
  \end{aligned}
  \end{equation}
and Sobolev spaces for $\ell \in \N$
\begin{equation} \label{eq2:Wspace}
  \begin{aligned}
	H^\ell_k(\R,\R^n) & := \{ v \in L^2_k(\R,\R^n) \cap H^\ell_{\mathrm{loc}}(\R,\R^n):\, \partial^j v \in L^2_k(\R,\R^n),\, 1\le j \le \ell \}, \\
	\| v \|_{H^\ell_k}^2 & := \sum_{j = 0}^\ell \| \partial^j v \|_{L^2_k}^2.
  \end{aligned}
  \end{equation}
With theses preparations we  state our main result.
\begin{theorem} \label{mainresult}
Let Assumption \ref{A1}-\ref{A4} be satisfied and let $m,k,\mu > 0$ be given with $\frac{9}{2} < m < m + \mu \le k$. Then there exist constants $\varepsilon_0, C_\infty > 0$, $K \ge 1$ such that the following statements hold for all initial perturbations $u_0 \in H^2_k$ which satisfy $\| u_0 \|_{H^1_{k + m + \mu}} < \varepsilon_0$: The initial-value problem \eqref{perturbsys} has a unique global solution $u$ which can be represented as
\begin{align*}
	u(t) = v_\star(\cdot - \tau(t)) + w(t), \quad t \in [0,\infty)
\end{align*}
for suitable $\tau \in C^1([0,\infty),\R)$ and $w \in C((0,\infty),H^2_k) \cap C^1([0,\infty), L^2_k)$. Further, there is an asymptotic phase $\tau_\infty = \tau_\infty(u_0) \in \R$ with $|\tau_\infty| \le C_\infty \| u_0 \|_{H^1_{k + m + \mu}}$ such that the
    following estimates hold for $t \ge 0$:
\begin{equation} \label{decayestimates}
\begin{aligned}
	& \| w(t) \|_{H^1_k} \le\frac{K}{(1+t)^{\frac{m^*-2}{2}}}\| u_0 \|_{H^1_{k + m + \mu}}, \quad |\tau(t) - \tau_\infty| \le \frac{K}{(1+t)^{\frac{m^*-4}{2}}} \| u_0 \|_{H^1_{k + m + \mu}}
\end{aligned}
\end{equation}
where $m^* = \lfloor m \rfloor + \max(0,2q-1)$ for $m = \lfloor m \rfloor + q$, $0 \le q < 1$.
\end{theorem}

In  Remark \ref{remark:exponent} we  explain in detail how the value of the exponent $m^*$ arises from  the interaction of linear and nonlinear terms.

\sect{Spectral analysis of the linearized operator and resolvent estimates}
\label{sec3}
In this section we discuss the spectral properties of the linearized operator $L$ from \eqref{LatL} considered as an operator on the algebraically weighted space $L^2_k$ for some $k \ge 0$. We will derive sharp estimates of the solution to the resolvent equation
\begin{align} \label{poly:resolvGl}
	(sI - L)v = r, \quad s \in \C,
\end{align}
where $r \in L^2_{\ell}$ with suitable $\ell \ge k$ and where $s$ lies in various subdomains of $\C$.
\subsection{Resolvent estimates for large $s$-values and the dispersion set} \label{sec3:0}
We begin with collecting basic properties of $L$ proved by standard estimates.

\begin{lemma} \label{poly:aprioriest}
Let the Assumptions \ref{A1}-\ref{A2} be satisfied and let $k \ge 0$. Then $L$ is a closed and densely defined linear operator on $L^2_k$ with $\D(L) = H^2_k$, i.e., $L \in \clos[L^2_k]$. Furthermore, there are constants $\varepsilon_0 , R_0, C > 0$ such that for all $s$ in 
\begin{align*}
  \Omega_0 := \Big\{ s \in \C: |s| \ge R_0, \, |\arg(s)| \le \frac{\pi}{2} + \varepsilon_0 \Big\}
\end{align*}
 and for every solution  $v \in H^2_k$ of \eqref{poly:resolvGl} with $r \in L^2_k$  the following estimate holds
\begin{align} \label{poly:largeresolest1}
	|s|^2 \| v \|_{L^2_k}^2 + |s| \| v_x \|_{L_k^2}^2 + \|v_{xx}\|_{L_k^2}^2 \le C \| r \|_{L^2_k}^2.
\end{align}
\end{lemma}

\begin{proof}
The estimate \eqref{poly:largeresolest1} is a standard energy estimate and can be shown with a slight modification of the proof of \cite[Lemma 3.1]{BeynDoeding22}. Details can be found in \cite[Ch.5.3]{Doeding}. The closedness of $L$ is then concluded from \eqref{poly:largeresolest1}. 
\end{proof}

\begin{figure}[h!] 
\centering
\begin{minipage}[t]{0.45\textwidth}
\centering
\includegraphics[scale=0.35]{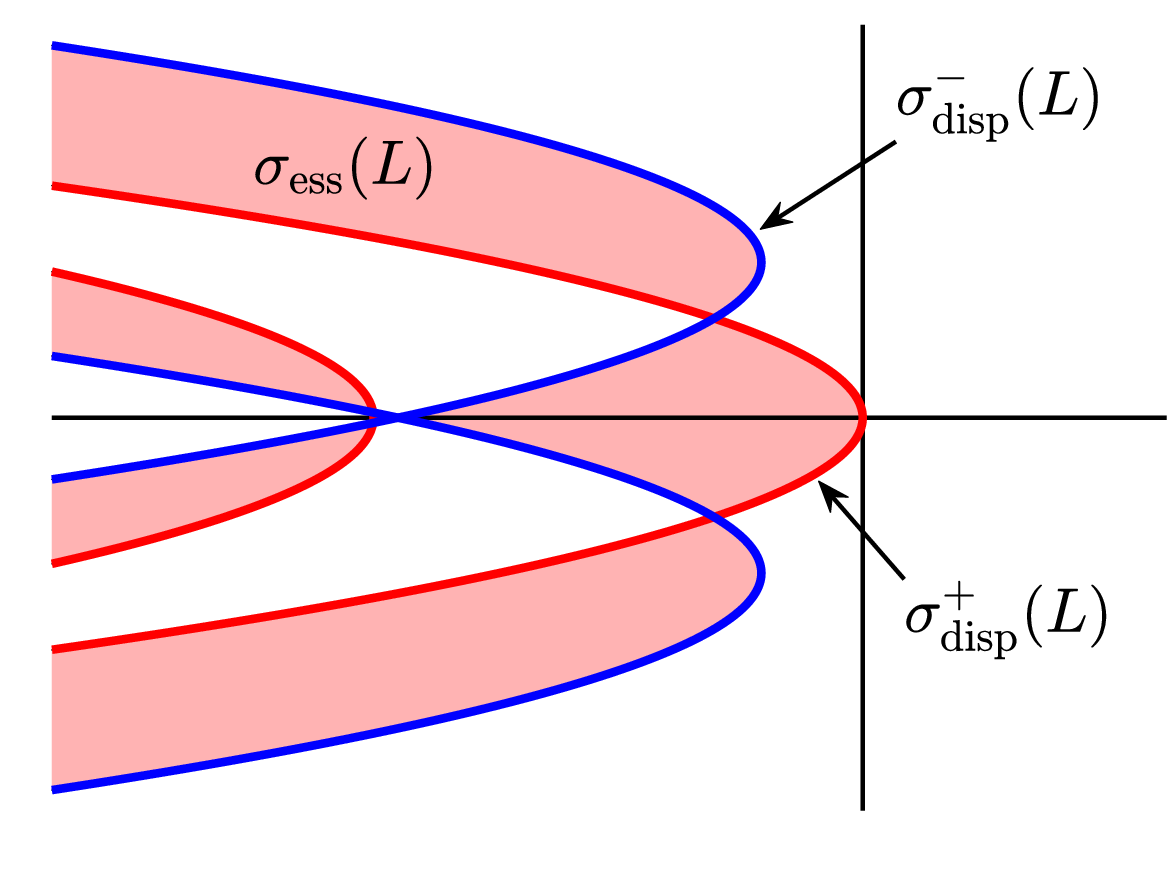}
\end{minipage}
\begin{minipage}[t]{0.45\textwidth}
\centering
\includegraphics[scale=0.35]{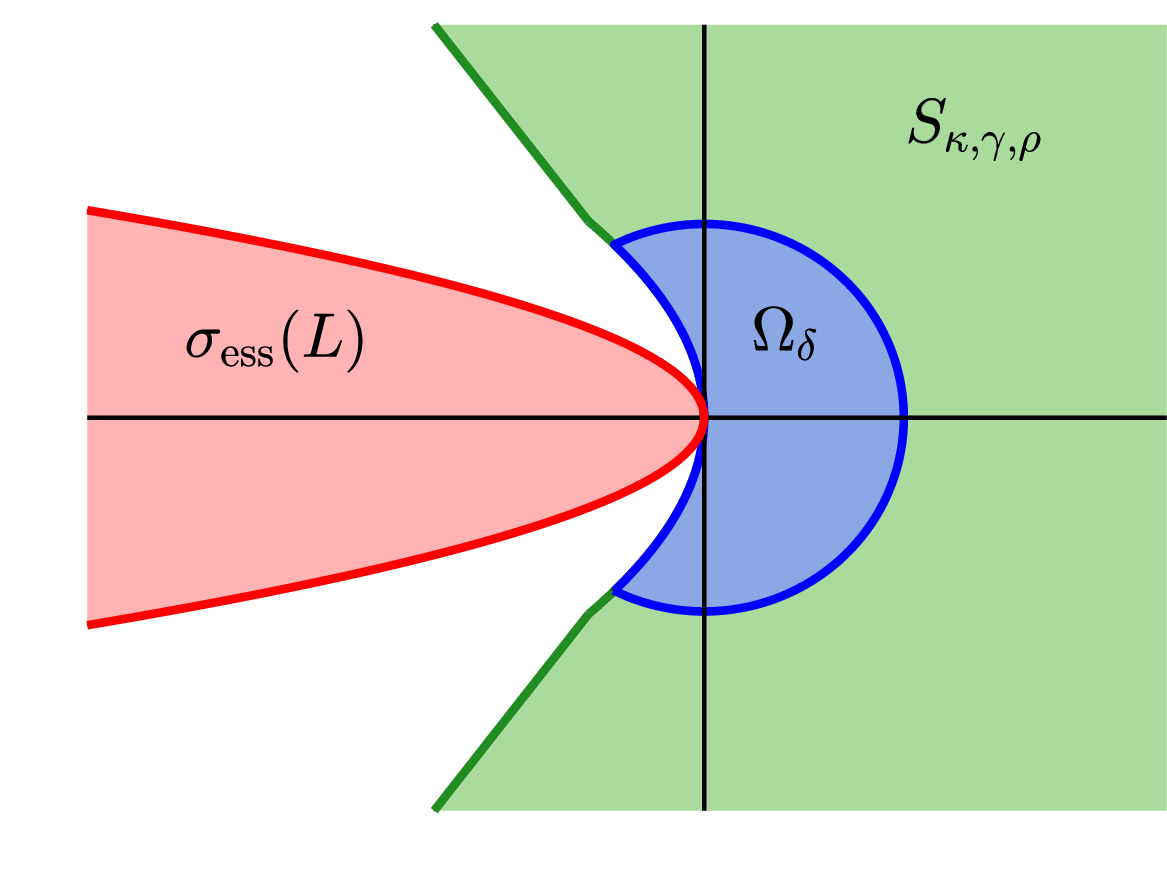}
\end{minipage}
\caption{Left: essential spectrum $\sigma_{\ess}(L)$ of the operator $L$ and the dispersion curves $\sigma^+_{\disp}(L)$ (red) and $\sigma^-_{\disp}(L)$ (blue). Right: rounded sector $S_{\kappa,\gamma,\rho}$ from Theorem \ref{Thm:essential} (green) and the $\delta$-crescent $\Omega_\delta := S_{\kappa,\gamma,\rho} \cap B_\delta(0)$ (blue).
} \label{essspec}
\end{figure}

As a next step we characterize the essential spectrum of $L \in \clos[L^2_k]$. In \cite{BeynDoeding22} the essential spectrum of $L$ was determined for $k = 0$ in exponentially weighted $L^2$-spaces. In general, there is a well developed spectral theory for second order differential operators on the real line in the context of traveling waves; we quote \cite{KapitulaPromislow} and \cite{Sandstede02} as general references. It is well known that the Fredholm properties of $L$ and therefore its essential spectrum is invariant under algebraic weights of the underlying $L^2$-space. In particular, the essential spectrum of $L \in \clos[L^2_k]$ coincides with the essential spectrum of $L \in \clos[L^2]$. This is a consequence of the invariance of the Fredholm index under compact perturbation \cite[Chapter IX]{EdmundsEvans} and can be shown along the lines of \cite[Lemma 3.2]{BeynDoeding22}. Another compact perturbation argument shows that the operator $sI - L \in \clos[L^2]$ is Fredholm for $s$ in the set
\begin{equation} \label{OmegaF}
\begin{aligned}
	& \Omega_F := \{ s \in \C: M_-(s) \text{ and } M_+(s) \text{ are hyperbolic} \}, \\
	&  M_\pm(s) = \begin{pmatrix}
	0 & I \\ A^{-1}(sI - C_\pm) & -cA^{-1}
\end{pmatrix}, \\
	& C_- = S_\omega + Df(0), \quad C_+ =  S_\omega + Df(v_{\infty})
\end{aligned}
\end{equation}
with Fredholm index $\mathrm{ind}(sI - L) = m_{\st}^+(s) - m_{\st}^-(s)$. The matrices $M_{\pm}(s)$ refer to
  the limits at $\pm \infty$ of the matrix in the corresponding first order system; see \eqref{firstvar} below. Hyperbolicity means that there are no eigenvalues on the
imaginary axis and $m_{\st}^\pm(s)$ denote the dimension of the stable subspaces of $M_\pm(s)$ for $s \in \Omega_F$ (see e.g. \cite[Lemma 3.1.10]{KapitulaPromislow}). Analyzing  the eigenvalue problem for $M_\pm(s)$ then yields $\Omega_F = \C \setminus \sigma_{\disp}(L)$ with the dispersion set
\begin{align*}
	\sigma_{\disp}(L)& := \sigma^+_{\disp}(L) \cup \sigma^-_{\disp}(L), \\
	\sigma^\pm_{\disp}(L)& := \{ s \in \C : \det (sI - D_\pm(\nu)) = 0 \text{ for some } \nu \in \R \}, \\
	D_\pm(\nu) & = -\nu^2 A + i \nu c I + C_\pm. 
\end{align*}
With these results in mind, the following theorem is a consequence of \cite[Theorem 3.5]{BeynDoeding22}.

\begin{theorem} \label{Thm:essential}
Let Assumptions \ref{A1}-\ref{A3} be satisfied. Then there are constants $\kappa, \gamma, \rho > 0$ such that the open set $\Omega_F$ has a unique connected component $\Omega_\infty$ satisfying
\begin{align*}
	S_{\kappa,\gamma,\rho} := \{ s \in \C  \backslash \{ 0 \} : \Re s \ge -\kappa \min(|\Im s|, \rho)^2 + \gamma \min (\rho - |\Im s|, 0) \} \subset \Omega_\infty.
\end{align*}
Further, for $s \in \Omega_\infty$ and $k \ge 0$ the operator $sI - L \in \clos[L^2_k]$ is Fredholm of index $0$, in particular, $\sigma_{\ess}(L) \subset \C \setminus \Omega_\infty$ holds.
\end{theorem}

\begin{proof}
  The claim follows from \cite[Theorem 3.5]{BeynDoeding22} and the arguments above. Let us emphasize that 
  Assumption \ref{A3} is the key condition which guarantees that such a rounded sector $S_{\kappa,\gamma,\rho}$
  pointing into the negative half-plane exists; see Figure \ref{essspec}.
\end{proof}

In Figure \ref{essspec} we show  the essential spectrum and the dispersion set for the QCGL with parameters $\alpha = \tfrac{1}{2}$ and $\beta_0 = -\tfrac{1}{4} + \tfrac{i}{2}$, $\beta_2 = 1 + i$, $\beta_4 = -1 + i$ in \eqref{QCGL}. In general, the dispersion set $\sigma_{\disp}(L)$ consists of four curves in $\C$ which lie in the complement of the rounded sector $S_{\kappa,\gamma,\rho}$ pointed at zero. One of the curves in $\sigma^+_{\disp}(L)$ is critical; it touches the
imaginary axis at the origin in a quadratic fashion so that $\sigma_{\ess}(L) \cap \overline{S_{\kappa,\gamma,\rho}} = \{ 0 \}$. As we will see later, this is due to the fact that one of the (spatial) eigenvalues of $M_+(s)$ (denoted by $\lambda(s)$ later in this section) crosses the imaginary axis when $s \in \C$ crosses the critical curve from right to left. Thus, the appearance  of eigenfunctions in the non-empty kernel $\ker(L)$ as well as the essential spectrum touching the imaginary axis prevent uniform bounds of the resolvent to hold for $s$ near the origin and for functions  $v$ in $H^1_k$ which solve  \eqref{poly:resolvGl}  with
$r \in L^2_k$. But, as we will show, restricting the class of right hand sides $r$ to a codimension $1$ subspace of $L^2_{k+1+\mu}$ for some $\mu > 0$ allows for such uniform bounds and sharp resolvent estimates close to the origin, but to the right of the dispersion set. 

\subsection{Spectral analysis via dichotomies and trichotomies}

In this subsection we characterize dichotomies and trichotomies of the first order system for $z = (v,v_x)^\top$ equivalent to \eqref{poly:resolvGl} and given by
\begin{equation} \label{firstvar}
\begin{aligned} 
  \M(s)z&=(\partial_x - M(s,\cdot))z = R, \\
  M(s,x) &= \begin{pmatrix}
	0 & I_2 \\ A^{-1}(sI-S_\omega - Df(v_\star(x))) & -cA^{-1}
	\end{pmatrix}, \quad R = \begin{pmatrix} 0 \\- A^{-1}r \end{pmatrix}.
\end{aligned}
\end{equation}
By $\S(s,\cdot,\cdot)$ we denote the solution operator of \eqref{firstvar} which is analytic in $s$, i.e., the function
$z(x) := \S(s,x,y) \xi$, $x\in \R$ solves the initial value problem $\M(s)z = 0$, $z(y) = \xi$ for $y \in \R$, $\xi\in \C^2$. The notion of dichotomies and trichotomies that we use throughout the paper is recalled in Appendix \ref{appendixA}. These properties will be crucial for the resolvent estimates in the subsequent sections. 

\begin{figure}[h!]
\centering
\begin{minipage}[t]{0.45\textwidth}
\centering
\includegraphics[scale=0.35]{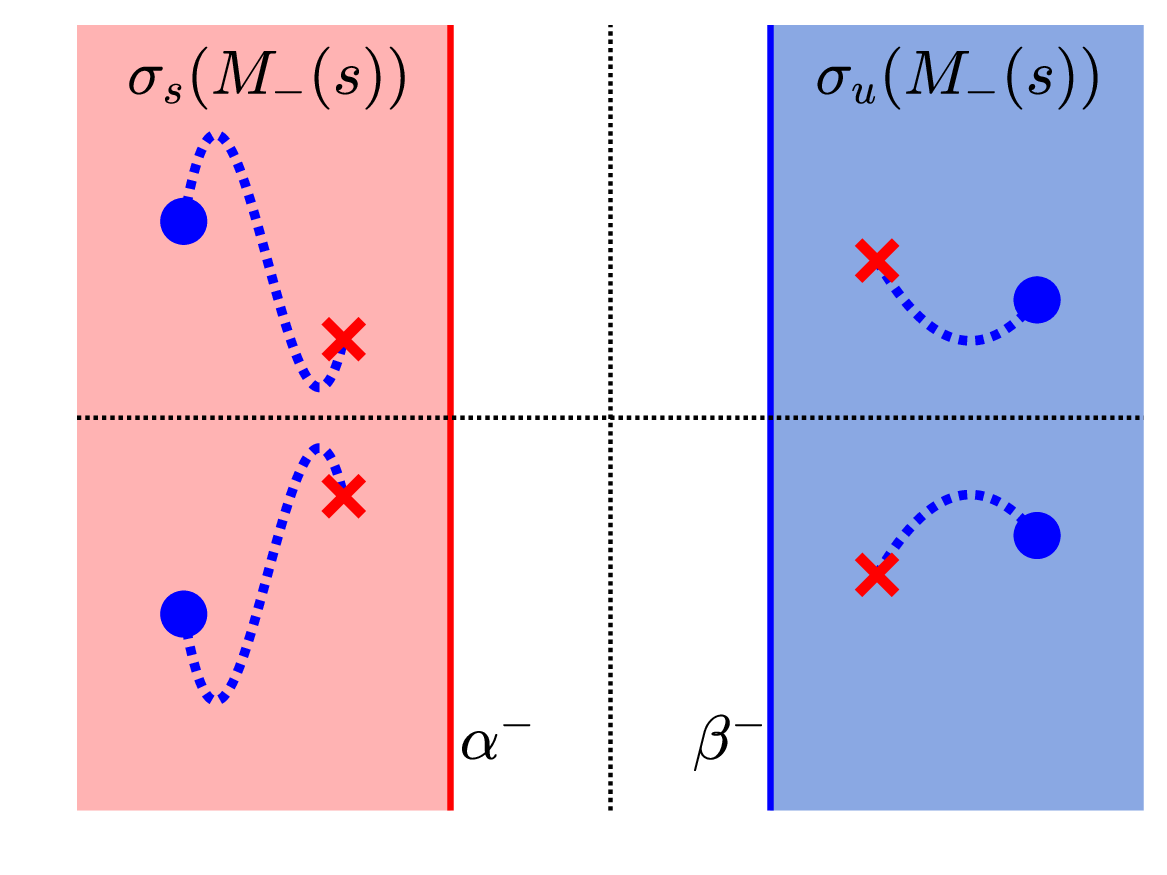}
\end{minipage}
\begin{minipage}[t]{0.45\textwidth}
\centering
\includegraphics[scale=0.35]{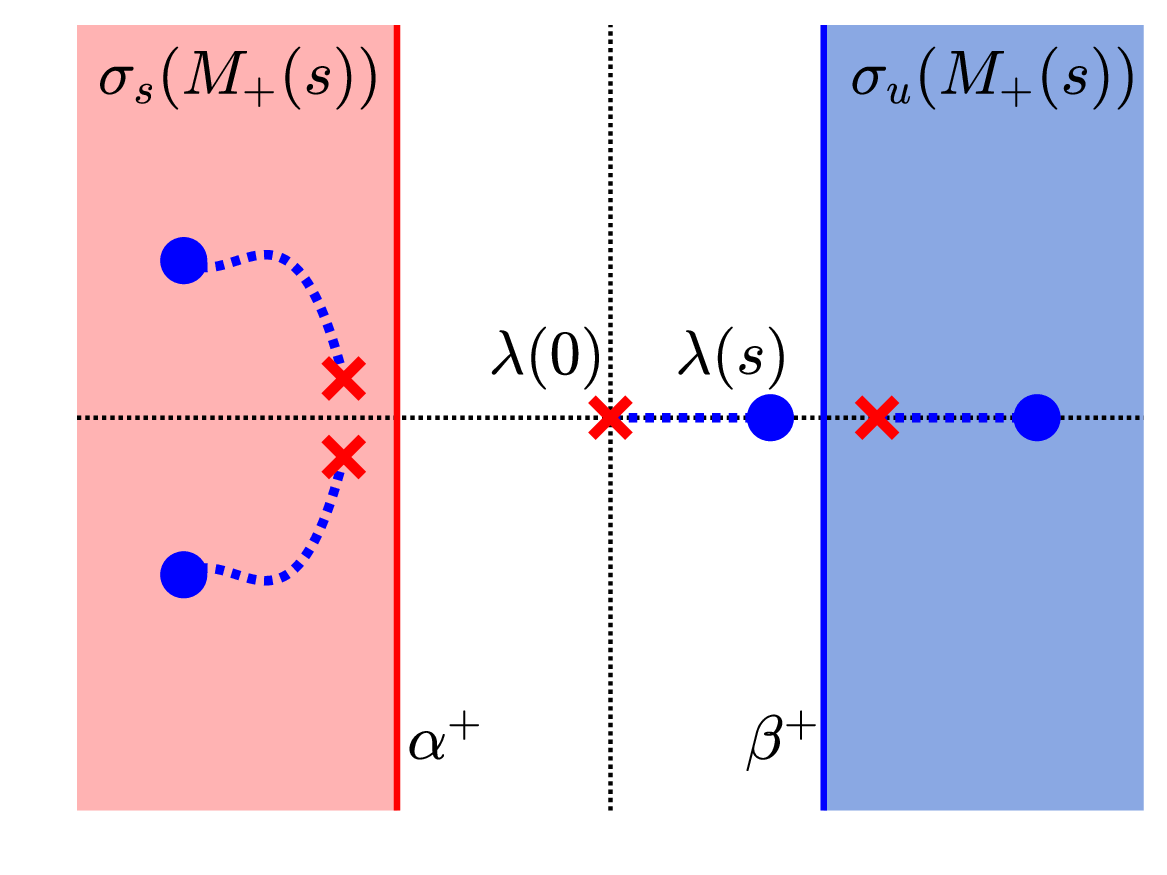}
\end{minipage}
\caption{Stable spectrum $\sigma_{\st}$ and unstable spectrum $\sigma_{\un}$ of $M_-(s)$ (left) and $M_+(s)$ (right) for $s \in (0,\delta) \subset \Omega_\delta$ (blue dots) and $s = 0$ (red crosses).} \label{figure-eigenvalues}
\end{figure}

  The first step is to analyze the eigenvalues of the limiting matrices $M_{\pm}(s)=\lim_{x \to \pm \infty}M(s,x)$ from \eqref{OmegaF}. In particular,
  when $s$ moves from $0$ to some $\delta>0$, the eigenvalues of $M_{\pm}(s)$  move as sketched
  in Figure \ref{figure-eigenvalues}: $M_{-}(s)$ has two stable and two unstable eigenvalues which do not cross
  the imaginary axis while $M_+(s)$ has two stable eigenvalues, an unstable one that stays real positive, and then a critical
  real one $\lambda(s)$ which moves from zero to the right. Since we allow multiple eigenvalues for the groups of two
  it is convenient to work with invariant subspaces and spectral subsets rather than eigenvectors and
  eigenvalues (a  viewpoint stressed in the monograph \cite{StewartSun90} when compared to
  the classical reference \cite[Ch.2]{Kato}). The following lemma 
  deals with  a specific perturbative situation which needs a separate proof.

\begin{lemma} \label{dichomatrices}
Let  Assumption \ref{A1}-\ref{A3} be satisfied. Then there exist $C, K,\delta, \kappa_{\star} > 0$ and $a^\pm < 0 < b^\pm$ such that the following statements hold for $s \in B_\delta(0)$:
\begin{enumerate}
\item[(i)] The matrix $M_-(s)$ is hyperbolic with stable subspace of dimension  $m_\st^-(s)=2$ and unstable subspace of
  dimension $m_\un^-(s) = 2$. The corresponding spectral projectors $\P^-_{\st}(s)$, $\P^-_{\un}(s)$ of rank $2$ are analytic in
  $s$ and satisfy
  \begin{equation} \label{eq3:dichminus}
  \begin{aligned}
    | \exp(xM_-(s))\P^-_{\st}(s)| & \le K e^{a^-x}, \quad x \ge 0, \\
    |\exp(xM_-(s))\P^-_{\un}(s)| & \le K e^{b^- x}, \quad x \le 0.
  \end{aligned}
  \end{equation}
  \item[(ii)] The spectrum of $M_+(s)$ can be decomposed into three subsets with associated spectral projectors
  $\P^+_{\st}(s)$, $\P^+_{\c}(s)$,  $\P^+_{\un}(s)$ which are of rank $m_\st^+(s) = 2$,  $m_\c^+(s) = m_\un^+(s) = 1$, respectively,
    and which depend analytically on $s$.  They satisfy the estimates
    \begin{equation} \label{eq3:dichplus}
   \begin{aligned}
    | \exp(xM_+(s))\P^+_{\st}(s)| & \le K e^{a^+ x}, \quad x \ge 0, \\
    |\exp(xM_+(s))\P^+_{\un}(s)| & \le K e^{b^+x}, \quad x \le 0.
   \end{aligned}
   \end{equation}
   The central projector is of the form $\P^+_{\c}(s) = v(s) w^{H}(s)$ where
   $M_+(s) v(s)=\lambda(s)v(s)$, $w^{H}(s)M_+(s)=\lambda(s) w^{H}(s)$ and $w^{H}(s)v(s)=1$ holds. Further,
   $\P_{\c}^+(s)$  satisfies 
   \begin{equation} \label{eq3:dichcentral}
     | \exp(xM_+(s)) \P^+_{\c}(s)| \le K e^{\Re \lambda(s) x} , \quad x \in \R.
   \end{equation}
      The eigenvalue  $\lambda(s)\in \C$ is simple and analytic in  $s$, and it has the properties
  \begin{enumerate}
\item $\lambda(s)\in \R$ for $s \in (-\delta,\delta)$,
\item$\lambda(0)=0$,  $\lambda'(0)>0$, $\lambda''(0) < 0$,
\item $0<|\lambda(s)|^2 \le C \Re \lambda(s)$ if $\Re s \ge -\kappa_{\star} |\Im s|^2$.
\end{enumerate}
\end{enumerate}
\end{lemma}
\begin{proof} From \eqref{SystemDef} and \eqref{OmegaF} we have
  \begin{align*}
    Df(u) &= g(|u|^2)+2g'(|u|^2)\begin{pmatrix} u_1^2 & u_1 u_2 \\ u_1 u_2 & u_2^2 \end{pmatrix}, \\
    Df(v_{\infty}) &= g(|v_{\infty}|^2)+2g'(|v_{\infty}|^2)\begin{pmatrix} |v_{\infty}|^2 & 0 \\ 0 & 0 \end{pmatrix},\\
    M_{\pm}(0)& =\begin{pmatrix}  0 & I \\ -A^{-1} C_{\pm} & - c A^{-1} \end{pmatrix},\quad
    C_-  = S_{\omega}+g(0), \\
     C_+& = \begin{pmatrix} \sigma_1 & 0 \\ \sigma_2 & 0 \end{pmatrix}, \text{ where }
    \sigma_j=2 g_j'(|v_{\infty}|^2)|v_{\infty}|^2 \, (j=1,2).
  \end{align*}
  Lemma \ref{appB:block}(i) applies to $M_-(0)$ since $A$ and $-C_-$ have a positive lower spectral bound due to Assumption \ref{A1}.
  Hence $M_-(0)$ is hyperbolic with
  two-dimensional stable and unstable subspace and corresponding projectors $\P^-_{\st}(0)$, $\P^-_{\un}(0)$. Then there is a norm $|\cdot|_0$ in $\R^m$ such that
  \begin{align*}
    |\exp(M_-(0))\P^-_{\st}(0)|_0 < 1, \quad |\exp(-M_-(0))\P^-_{\un}(0)|_0 <1 .
  \end{align*}
  This property persists for $|s|$ small with Riesz projectors analytic in $s$ and  given by
  \begin{equation} \label{eq3:sproj}
  \begin{aligned}
    \P^-_{\st}(s)& =\frac{1}{ 2 \pi i} \int_{\Gamma^-} (zI - M_-(s))^{-1} dz, \quad
    \P^-_{\un}(s) =\frac{1}{ 2 \pi i} \int_{\Gamma^+} (zI - M_-(s))^{-1} dz,
  \end{aligned}
  \end{equation}
  where $\Gamma^{\pm}$ are semicircles with sufficiently large radius $R$ enclosing the stable and unstable subsets
  and parameterized by
  \begin{align*}
    \gamma_{\pm}(\tau) = \begin{cases} \pm R^{-1} \mp i R (2\tau-1), & 0 \le \tau \le 1, \\
      \pm R^{-1}\pm i R e^{ i \pi \tau}, & 1 \le \tau \le 2.
    \end{cases}
     \end{align*}
   In particular, there exist $\delta>0$ and $a^-<0<b^-$ such that we have for $|s|\le \delta$  
\begin{align} \label{eq3:normsmall}
    |\exp(M_-(s))\P^-_{\st}(s)|_0 \le e^{a^-}<1, \quad |\exp(-M_-(s))\P^-_{\un}(s)|_0 \le e^{-b^-}<1  .
  \end{align}
Taking powers $\exp(n M_-(s))$ for $n\in \N$ with this norm and filling the gaps $(n,n+1)$  by continuity then leads to the
estimate  \eqref{eq3:dichminus} using the equivalence of norms $| \cdot |_0$ and $| \cdot |$.

The proof of (ii) is somewhat more involved. By Assumption \ref{A2} we have $\sigma_1<0$ so that $-C_+$ has a simple eigenvalue
$\lambda(0)=0$ with eigenvector $\begin{pmatrix} 0 & 1 \end{pmatrix}^{\top}$ and another simple eigenvalue $-\sigma_1>0$.
Using further $c>0$, Lemma  \ref{appB:block}(ii) ensures that $M_+(0)$ has a two-dimensional stable subspace, a one-dimensional
unstable subspace and a one-dimensional critical subspace which belongs to the zero eigenvalue.
The critical right eigenvector $v_0$ and left eigenvector $w_0$ of $M^+(0)$ are given by
\begin{align}\label{eq3:v0w0}
  v_0= ( 0, 1, 0, 0)^\top,  \quad w_0 = \begin{pmatrix}y \\
    c^{-1}A^{\top}y \end{pmatrix}, \quad
  y = \begin{pmatrix} - \sigma_2 \\ \sigma_1 \end{pmatrix}.
\end{align}
The spectral projectors $\P^+_{\st}(0)$ of rank $2$ and  $\P^+_{\un}(0)$ of rank $1$
can be continued analytically for $|s|\le \delta$ as in
\eqref{eq3:sproj}. Similarly, the rank $1$ projector $\P_{\c}^+(0)= \sigma_1^{-1}v_0 w_0^{\top}$
is continued analytically by
\begin{align*}
  \P^+_{\c}(s)& =\frac{1}{ 2 \pi i} \int_{\partial B_{2\delta}(0)} (zI - M_+(s))^{-1} dz,
  \end{align*}
if $\delta $ and $R>0$ are chosen  such that the curves $\Gamma^{\pm}$ and $\partial B_{2\delta}(0)$ do not overlap.
For the stable and the unstable projection we obtain contraction as in \eqref{eq3:normsmall} and then the exponential estimates
\eqref{eq3:dichplus}. Another method of determining the unique central projector is via the analytic implicit function theorem
applied to the mapping
\begin{align} \label{eq3:Tdef}
  T:\C^4 \times \C \times \C \to  \C^4 \times \C, \quad (v,\lambda,s)\mapsto
  \begin{pmatrix}M_+(s) v - \lambda v \\ w_0^{\top} v - \sigma_1 \end{pmatrix}
\end{align}
for which $T(v_0,0,0)=0$ holds. The derivative w.r.t. $(v,\lambda)$
\begin{align*}
  D_{(v,\lambda)}T(v_0,0,0)= \begin{pmatrix} M_+(0)  & -v_0 \\ w_0^{\top} & 0 \end{pmatrix}
  \end{align*}
is invertible since $0$ is a simple eigenvalue of $M_+(0)$. The local solution $(v(s),\lambda(s))$, $|s|\le \delta$ of $T(v,\lambda,s)=0$
satisfies $(v(0),\lambda(0))=(v_0,0)$ and solves the
eigenvalue problem
\begin{equation} \label{eq3:evs}
  M_+(s)v(s) = \lambda(s)v(s), \quad w_0^{\top}v(s)=\sigma_1.
  \end{equation}
  From a suitable
adjoint system we obtain analytic left eigenvectors $\tilde{w}(s)$ normalized by $v_0^{\top}\tilde{w}(s)=\sigma_1$.  Setting
$w(s) = p(s)^{-1}\tilde{w}(s)$, where $\bar{p}(s)= \tilde{w}^{H}(s)v(s)$, then leads to the representation $\P_{\c}(s)= v(s) w(s)^{H}$ with $w(s)^{H}v(s)=1$. The estimate \eqref{eq3:dichcentral} follows directly from this representation.

Assertion (a) follows since  the real implicit function theorem equally applies to $T$ from \eqref{eq3:Tdef} when restricted
to $\R^6$. By uniqueness we then have $\lambda(s)\in \R$ for $s \in [-\delta,\delta]$ as well as $\lambda'(0),\lambda''(0) \in \R$.
We differentiate \eqref{eq3:evs} at $s=0$ and find using \eqref{eq3:v0w0}
\begin{align} \label{eq3:impdiff}
  M_+'(0)v_0 + M_+(0)v'(0)= \lambda'(0)v_0, \quad w_0^{\top} v'(0)=0.
\end{align}
Multiplying by $w_0^{\top}$ yields $\lambda'(0)\sigma_1=c^{-1}y^{\top}A A^{-1}\left( \begin{smallmatrix} 0 \\ 1
  \end{smallmatrix} \right)=c^{-1} \sigma_1$,
hence $\lambda'(0)=\frac{1}{c} >0$. Moreover, solving \eqref{eq3:impdiff} for $v'(0)$ and using Assumption \ref{A3} yields $v'(0)= (0, q, 0, c^{-1})^{\top}$ where
\begin{align*} 
c^2 \sigma_1q= \alpha_1 \sigma_1 + \alpha_2 \sigma_2=-2|v_{\infty}|^2\left(\alpha_1 g'_1(|v_{\infty}|^2)+\alpha_2 g'_2(|v_{\infty}|^2\right)
    >0.
\end{align*}
Differentiating \eqref{eq3:evs} twice at $s=0$ and multiplication by $w_0^{\top}$ leads to
\begin{align*}
  \lambda''(0) \sigma_1= 2 w_0^{\top}M'_+(0)v'(0)= 2 c^{-1}y^{\top}A A^{-1}\begin{pmatrix} 0 \\ q \end{pmatrix}=
  \frac{2\sigma_1 q}{c},
\end{align*}
hence $\lambda''(0)<0$. Finally, we establish (c) by using the Taylor expansion
\begin{align} \label{eq3:explambda}
  \lambda(s)= \lambda'(0)s + \frac{1}{2} \lambda''(0)s^2 + \cO(|s|^3).
\end{align}
Since $\lambda'(0)>0$ we have $C_1|s| \le |\lambda(s)| \le C_2|s|$ for $|s|\le \delta$ and some $C_1,C_2,\delta>0$.
This implies  $\lambda(s)\ge C \lambda(s)^2$ for  real $s \in (0,\delta]$.
  Otherwise $\tau:= \Im(s)\neq 0$ and we can write $s=i\tau - a \tau^2$ where $|\tau|, |a|\tau^2 \le |s|\le \delta$.
 In the following we consider  $a \le \kappa_{\star}$ where $\kappa_{\star}>0$ is still to be determined.
 We insert $s$ into \eqref{eq3:explambda} and obtain
 \begin{align*}
   \Re \lambda(s) &= \tau^2\gamma(a,\tau), \\
   \gamma(a,\tau) &= -a \lambda'(0) + \tfrac{1}{2}\lambda''(0)a^2\tau^2
   -\tfrac{1}{2} \lambda''(0)  + \cO(\delta(1+a^2 \tau^2)).
 \end{align*}
 If $0 \le a \le \kappa_{\star}$ we find for $\delta,\kappa_{\star}$ sufficiently small
 \begin{align*}
   \gamma(a,\tau) & = \tfrac{1}{2}|\lambda''(0)|(1- a^2 \tau^2)- a \lambda'(0) + \cO(\delta(1+a^2\tau^2)) \\
   &  \ge \tfrac{1}{2}|\lambda''(0)|(1- \kappa_{\star}\delta) - \kappa_{\star}\lambda'(0) -C\delta(1+\kappa_{\star} \delta)
   \ge \gamma_0 >0.
 \end{align*}
Since $|s|^2 \le \tau^2(1+\kappa_{\star}a\tau^2) \le \tau^2(1+\kappa_{\star}\delta)$  there exists a constant  $C>0$ such that
 \begin{align*}
   \Re \lambda(s)\ge \gamma_0 \tau^2 \ge C\tau^2(1+\kappa_{\star}\delta) \ge C|s|^2 \ge \tfrac{C}{C^2_2} |\lambda(s)|^2.
 \end{align*}
 In case $a\le 0$ we obtain for small $\delta\le 1$ and suitable $C,C_3>0$
 \begin{align*}
   \gamma(a,\tau)&=|a|(\lambda'(0)+ \tfrac{1}{2}\lambda''(0)|a|\tau^2) - \tfrac{1}{2}\lambda''(0) + \cO(\delta(1+a^2\tau^2))\\
   & \ge |a|(\lambda'(0)-|\lambda''(0)|\tfrac{\delta}{2}) - \tfrac{1}{2}\lambda''(0)- C_3\delta(1+a^2 \tau^2) \\
   & \ge \left( \tfrac{1}{2 \delta}\lambda'(0)-C_3 \delta\right) a^2 \tau^2 + \tfrac{1}{2}(|\lambda''(0)|-2 C_3 \delta)
   \ge C(1+ a^2 \tau^2).
 \end{align*}
 Finally, this leads to
 \begin{align*}
 \Re \lambda(s)\ge C \tau^2(1+a^2 \tau^2)= C|s|^2 \ge \tfrac{C}{C_2^2} |\lambda(s)|^2,
 \end{align*}
which finishes the proof.
\end{proof}

Using the roughness theorems \cite[Proposition 4.1]{Coppel} and Theorem \ref{RoughTricho} we establish analogous
estimates for the variable coefficient operator $\M$ from \eqref{firstvar}.

\begin{lemma} \label{dichoL}
  Let Assumptions \ref{A1}-\ref{A3} be satisfied, let $\delta$, $a^{\pm}$, $b^{\pm}$ be given  by Lemma \ref{dichomatrices},
  and take any  $\alpha^{\pm},\beta^{\pm}$ with 
  $a^{\pm}<\alpha^\pm <  0 < \beta^\pm< b^{\pm}$. Then the following statements hold for $s \in B_{\delta}(0)$ and $\delta$ sufficiently small:
\begin{enumerate}
\item[(i)] The operator $\M(s)$ has an exponential dichotomy on $\R_-$ with data \\ $(K,\alpha^-,\beta^-,P_{\st,\un}^-(s,\cdot))$ where for each $x \in \R_-$ the projectors $P_{\st,\un}^-(s,x)$ depend analytically on $s$ and have rank $m_\st^-(s) = m_\un^-(s) = 2$.
\item[(ii)] The operator $\M(s)$ has an exponential trichotomy on $\R_+$ with data \\ $(K,\alpha^+,\nu(s),\beta^+,P_{\st,\c,\un}^+(s,\cdot))$ where $\nu(s)=\Re \lambda(s)$ is given by Lemma \ref{dichomatrices} (ii) and the projectors $P_\kappa^+(s,x), \kappa = \un,\c,\st, x \ge 0$ have rank $m_\st^+(s) = 2$ and $m_\c^+(s) = m_\un^+(s) = 1$. The exponents  satisfy $\alpha^+ < \nu(s) < \beta^+ $ and the following estimates hold:
\begin{equation} \label{trichoplus}
\begin{aligned}
	& |\S(s,x,y) P_\st^+(s,y)| \le K e^{\alpha^+(x-y)}, \, |\S(s,x,y) P_\c^+(s,y)| \le K e^{\nu(s)(x-y)}, \, x \ge y, \\
	& |\S(s,x,y) P_\un^+(s,y)| \le K e^{\beta^+(x-y)}, \, |\S(s,x,y) P_\c^+(s,y)| \le K e^{\nu(s)(x-y)}, \, x \le y.
\end{aligned}
\end{equation}
\end{enumerate}
\end{lemma}

\begin{proof}
By \cite[Theorem 2.5]{BeynDoeding22} there are $C,\mu_* > 0$  such that we have for $s \in B_{\delta}(0)$ 
\begin{align*}
	|M(s,x) - M_-(s)| \le |A^{-1}| |Df(v_\star(x)) - Df(0)| \le C e^{\mu_* x}, \quad x\le 0.
\end{align*}
Then the first assertion is a consequence of Lemma \ref{dichomatrices} and \cite[Proposition 4.1]{Coppel}. For the second assertion recall the ordinary exponential trichotomy of $\partial_x-M_+(s)$ on $\R_+$. Let $w(s), v(s) \in \C^n$ be the left and right eigenvectors of $M_+(s)$ associated with the eigenvalue $\lambda(s)$ from Lemma \ref{dichomatrices}. Then $\varphi(s,x) = e^{M_+(s)x} v(s)$ solves $\varphi_x-M_+(s)\varphi = 0$ and  $\psi(s,x) = e^{-M_+(s)^H x} w(s)$ solves the adjoint system $- \psi_x - M_+(s)^H\psi = 0$ on $\R_+$. Further, the central projector is given by  $\P_\c^+(s) = v(s) w(s)^H = \varphi(s,x) \psi(s,x)^H$ and the functions $e^{- \nu(s) x}|\varphi(s,x)|$ and $e^{\nu(s) x}|\psi(s,x)|$ are constant w.r.t. $x$. By \cite[Theorem 2.5]{BeynDoeding22} there exist $C,\mu_* > 0$  such that for all  $s \in B_{\delta}(0)$
\begin{align*}
	|M(s,x) - M_+(s)| \le |A^{-1}| |Df(v_\star(x)) - Df(v_\infty)| \le C e^{-\mu_* x}, \quad x \ge 0.
\end{align*}
Applying Theorem \ref{RoughTricho} we conclude that $\M(s)$ has an ordinary exponential trichotomy on $\R_+$ and we can choose the exponents $\alpha^+$, $\beta^+$ arbitrarily close to $a^+, b^+$ such that the estimates \eqref{trichoplus} hold.
\end{proof}
\subsection{Resolvent estimates for the linearized operator $L$}

In this subsection we study the resolvent equation \eqref{poly:resolvGl} of the linearized operator for right hand sides $r \in L^2_\ell$ with suitable $\ell \ge k$ and for $s \in \C$ in a suitable neighborhood of the origin. This neighborhood has
the shape of a crescent (see Figure \ref{essspec}) with a parabolic section determined by the value of $\kappa$ from
the rounded sector $S_{\kappa,\gamma,\rho}$ in Theorem \ref{Thm:essential} and by $\kappa_{\star}$ from Lemma \ref{dichomatrices}(ii)(c). We assume $\kappa \le \kappa_{\star}$ w.l.o.g. and define the $\delta$-crescent by 
\begin{align} \label{crescent}
	\Omega_\delta := S_{\kappa,\gamma,\rho} \cap B_\delta(0).
\end{align}
In the following we take $\delta \in (0,\rho)$ sufficiently small and we recall
that $0 \notin \Omega_\delta$ by the definition of $S_{\kappa,\gamma,\rho}$. 
We aim to show uniform bounds of the resolvent $(sI - L)^{-1}$ for $s \in \Omega_\delta$ w.r.t. suitable norms. Since the essential spectrum of $L$ touches the imaginary axis quadratically at the origin (Theorem \ref{Thm:essential}) a singularity of the resolvent of $L$ occurs at $s = 0$. In particular, this prevents us from showing uniform bounds for the resolvent on $L^2_k$ for $s \in \Omega_\delta$. Instead, we assume a stronger decay of the right hand side in \eqref{poly:resolvGl}, i.e., $r \in L^2_{k + 1 + \mu}$ for some $\mu > 0$, and show uniform bounds in $\Omega_\delta$ of the resolvent $(sI - L)^{-1}$ as an operator from $L^2_{k+1+\mu}$ to $L^2_k$. In addition, we resolve the order of the singularity caused by the essential spectrum in dependence on $0 \le q < 1$ when $r \in L^2_{k+q}$. The approach uses the exponential dichotomy and exponential trichotomy of the first order operator $\M(s)$ from Lemma \ref{dichoL} and the following estimates.
   
\begin{lemma} \label{est}
The following statements hold:
\begin{enumerate}
\item \label{item1} For every $k \ge 0$ and $0 \le q \le 1$ with either $k > 0$ or $q < 1$ there exists $C = C(k,q) > 0$ such that for all $\beta > 0$ and $x \ge 0$:
     \begin{equation} \label{eq3:est1}
     \eta^k(x) \int_x^{\infty} \eta^{-(k+q)}(y)e^{\beta(x-y)} dy \le C \beta^{q - 1}.
   \end{equation}
The estimate \eqref{eq3:est1} still holds in case $q = 1$ and $\beta = 0$.

\item \label{item3} For every $0 \le q < 1$ there is $C = C(q) > 0$ such that for all $\beta > 0$ and $x \ge 0 $
\begin{equation} \label{eq3:est4}
	\int_0^x \eta^{-q}(y) e^{\beta(x-y)} dx \le C \beta^{q-1}.
\end{equation}

\item \label{item2} For every $\beta_0 > 0$ and $k \ge 1$ there exists  $C = C(\beta_0,k) > 0$ such that for all $0 < \beta \le \beta_0$ and $x \ge 0$:
      \begin{equation} \label{eq3:est3}
      \eta^k(x) \int_0^{x} \eta^{-k}(y)e^{\beta(y-x)} dy \le C \begin{cases} \beta^{-k}, & k >1, \\
        |\log(\beta)|\beta^{-1}, & k=1. \end{cases}
    \end{equation}

\end{enumerate}
\end{lemma}

The proof of Lemma \ref{est} is given in Appendix \ref{appendixC}.

\begin{remark} \label{rem:Kapitula}
Lemma \ref{est} is an extension and a correction of  \cite[Lemma 3.2]{Kapitula91}, \cite[Lemma 3.3]{Kapitula94}. The extension concerns $k$ and $q$ to be noninteger in \eqref{eq3:est1}, \eqref{eq3:est3}, the new estimate \eqref{eq3:est4}, and the use of the weight function $\eta$ which avoids a case-by-case analysis of $x\le 1$ and $x \ge 1$. The correction concerns the estimate \eqref{eq3:est3} which is claimed there to hold with $C \beta^{-2}$ for all $k \in \N$.  For $k<2$ our estimate is an improvement, while for $k >2$ it corrects a flaw in  \cite[Lemma 3.2]{Kapitula91}. Let us further note that the forthcoming estimates will use counterparts of \eqref{eq3:est1} and \eqref{eq3:est3} on $\R_-$ which
are obtained by reflection.
\end{remark}

Apart from the singularity of the resolvent caused by the essential spectrum, another difficulty appears since $L$ has a nontrivial kernel $\ker(L) = \mathrm{span}( v_{\star,x} )$  caused by the translation invariance of \eqref{rEvo}. Roughly speaking, there is an "eigenvalue" hidden in the essential spectrum at the origin. The strategy then is to
split the space $L^2_{k}$, $k \ge 1$ into $\ker(L)$ and into its orthogonal complement. In this way
we obtain uniform resolvent estimates on the complement while simultaneously keeping track of the singular growth of the resolvent on $\ker(L)$. \\

\begin{proposition} \label{propEV}
  Let Assumption \ref{A1}-\ref{A4} be satisfied and let $k \ge 1$. Then there exists $\delta > 0$ and for every $0 \le q < 1$, $\mu>0$ a constant $C = C(k,q,\mu)$  such that the system \eqref{firstvar} has a unique solution $z(s) \in H^1_k$
  for all $s \in \Omega_\delta$ and $R \in L^2_{k}$ . This solution satisfies the estimates
\begin{alignat}{2} 
\| z(s) \|_{L^2_{k}} & \le C \Big(\frac{1}{|s|} \| Q_k R \|_{L^2_{k}} + |s|^{2q - 2} \|(I-Q_k)R\|_{L_{k+q}^2}\Big) \quad && \text{if } R \in L^2_{k + q}, \label{Yestnew2} \\
	\| z(s) \|_{L^2_{k}} & \le C \Big(\frac{1}{|s|} \| Q_k R \|_{L^2_{k}} + \|(I-Q_k)R\|_{L_{k+1+\mu}^2}\Big) \quad && \text{if } R \in L^2_{k + 1 + \mu}, \label{Yestnew}
\end{alignat}
where the map $Q_k:L_{k}^2 \to \mathrm{span}(\varphi)$ is given by
\begin{align*}
  Q_kR= \varphi(\psi,R)_{L^2}, \quad \varphi=\begin{pmatrix}v_{\star,x} & v_{\star,xx} \end{pmatrix}^{\top} 
\end{align*}
and $\psi$ is the unique bounded solution of the adjoint system
\begin{align}\label{eq3:adjoint}
  (\partial_x + M(0,\cdot)^{\top}) \psi = 0 \text{ in } \R, \quad -(\psi, \partial_sM(0,\cdot)\varphi)_{L^2}=1.
\end{align}
If $R \in L^2_{k + 1 +\mu}$ with $Q_kR=0$ then $z(s)$, $s \in \Omega_{\delta}$ can be continuously extended by some $z(0)\in H_k^1$ which solves \eqref{firstvar} for $s=0$ and
which satisfies $\|z(0)\|_{L_k^2}\le C \|R\|_{L^2_{k+1+\mu}}$.
\end{proposition}
\begin{remark}Writing $(\psi,R)_{L^2}$ is a slight abuse of notation since $\psi$ is only bounded. However,  the integral of $\psi^{\top}R$ exists for  $R \in L^2_k$ and  $\|Q_kR\|_{L^2_{k}}\le C \|R\|_{L^2_k}$ holds; see \eqref{eq3:psiR}.
  \end{remark}

\begin{proof}
  {\bf Step 1: Dichotomies and Trichotomies}\\
  In the following we frequently use the dichotomy resp. trichotomy of $\M(s)$  on  $\R_-$ resp. $\R_+$ provided by
  Lemma \ref{dichoL}. Let us  further note that the $s$-derivatives of solutions inherit the exponential estimates by Cauchy's theorem, i.e., for $|s| \le \delta \le \frac{\delta_0}{2}$ and sufficiently small $\delta_0$, we have for $x \ge y \ge 0$ 
  \begin{equation} \label{eq3:sderiv}
      |\partial_s\big(\S(s,x,y)P_{\st}^+(s,y)\big)| = \Big| \frac{1}{2 \pi i} \int_{|\tau|= \delta_0} \frac{\S(\tau,x,y)P_{\st}^+(\tau,y)}{(\tau-s)^2} d\tau \Big| \le \frac{4C}{\delta_0} e^{\alpha^+(x-y)},
   \end{equation}
   and for $x \le y \le  0$ 
   \begin{equation} \label{eq3:sderiv2}
      |\partial_s\big( \S(s,x,y)P_{\un}^-(s,y)\big)| = \Big| \frac{1}{2 \pi i} \int_{|\tau|= \delta_0}
      \frac{\S(\tau,x,y)P_{\un}^-(\tau,y))}{(\tau-s)^2} d\tau \Big| \le \frac{4C}{\delta_0} e^{\beta^-(x-y)}.
   \end{equation}
  Analogous estimates hold for the complementary projectors. In the following we prove the assertion for $s \in \Omega_{\delta}$.
  Let $V_{\st}^+(0,0),V_{\un}^-(0,0)\in \C^{4,2}$ be matrices such that 
  $\ran(P_{\st}^+(0,0)) =\ran(V_{\st}^+(0,0))$, $\ran(P_{\un}^-(0,0)) =\ran(V_{\un}^-(0,0))$ and define for $s \in B_{\delta}(0)$
  \begin{equation*}
    \begin{aligned}
      V_{\st}^+(s,0)&=P_{\st}^+(s,0)V_{\st}^+(0,0), \quad
     V_{\un}^-(s,0)= P_{\un}^-(s,0)V_{\un}^-(0,0), \\
     \Phi(s)&=(V_{\st}^+(s,0), V_{\un}^-(s,0)) \in \C^{4,4}.
     \end{aligned}
  \end{equation*}
  By Lemma \ref{dichoL} the matrix functions
  \begin{equation} \label{eq3:Vsol}
    V_{\st}^+(s,x)=\S(s,x,0)V_{\st}^+(s,0), \quad V_{\un}^-(s,x)=\S(s,x,0)V_{\un}^-(s,0)
  \end{equation}
  solve the homogenous equation \eqref{firstvar} and satisfy the exponential estimates
  \begin{equation*}
    | V_{\st}^+(s,x)| \le C e^{\alpha^+x} \; (x \ge 0), \quad |V_{\un}^-(s,x)| \le C e^{\beta^-x} \; (x \le 0).
  \end{equation*}
  Further, by Assumption \ref{A4} we have $\ker(L)= \mathrm{span}(v_{\star,x})$, hence there exist $v^{\pm} \in \R^2$ with
  \begin{equation*}
    \begin{aligned}
      \mathrm{span}(\varphi(0)) & = \ran(V_{\st}^+(0,0)) \cap \ran(V_{\un}^-(0,0)), \quad
       \ker(\Phi(0)) =\mathrm{span}(v_0),
      \\
       \varphi(0)&= \begin{pmatrix} v_{\star,x}(0) \\ v_{\star,xx}(0) \end{pmatrix}
      =V_{\st}^+(0,0)v^+= V_{\un}^-(0,0)v^-,  \quad v_0=\begin{pmatrix}v^+ \\ - v^- \end{pmatrix}.
    \end{aligned}
  \end{equation*}
   There also  exists some $\psi_0 \in \R^4$ with
   $  \psi_0^{\top}\Phi(0)=0$, $\ran(\Phi(0)) = \psi_0^{\perp}$. The solution operator of the adjoint $\M^{*}(s)=\partial_x+ M(s,\cdot)^{H}$ is given by
  \begin{align*}
    \S^{*}(s,x,y) = \S(s,y,x)^H, \quad s \in \C, x,y \in \R.
  \end{align*}
  It has an exponential dichotomy  with data 
\begin{equation*}  
  (K,- \beta^-(s), - \alpha^-(s), P_{\un}^{-}(s,\cdot)^H, P_{\st}^{-}(s,\cdot)^H) \quad \text{on } \R_-
\end{equation*}  
and, likewise, it has an exponential trichotomy with data
  \begin{equation} \label{eq3:adjtri}
  (K,-\beta^+(s),- \nu(s),-\alpha^{+}(s), P_{\un}^{+}(s,\cdot)^H,P_{\c}^{+}(s,\cdot)^H, P_{\st}^{+}(s,\cdot)^H)  \quad \text{on } \R_+.
 \end{equation}

  {\bf Step 2: The simple eigenvalue} \\
  From $\ran(\Phi(0))= \psi_0^{\perp}$ we have
  \begin{align*}
  \mathrm{span}( \psi_0) & = \big(\ran(P_{\st}^+(0,0)) + \ran(P_{\un}^-(0,0))\big)^{\perp}= \ran(P_{\st}^+(0,0))^{\perp} \cap \ran(P_{\un}^-(0,0))^{\perp}\\
  & = \ker(P_{\st}^{+}(0,0)^\top) \cap \ker(P_{\un}^{-}(0,0)^\top)=
  \ran(P_{\c \un}^{+}(0,0)^\top)  \cap \ran(P_{\st}^{-}(0,0)^\top),
  \end{align*}
  where we abbreviate $P_{\c \un}^+= P_{\c}^+ + P_{\un}^+$. In view of \eqref{eq3:adjtri} the space of bounded solutions of the system \eqref{eq3:adjoint}
  is one-dimensional and spanned by $\psi(x)=\S^{*}(0,x,0)\psi_0$. Further, $\psi$ satisfies the estimate
   \begin{equation} \label{eq3:psiest}
    |\psi(x)| \le K |\psi_0| \quad(x\ge 0), \qquad |\psi(x)| \le Ke^{\beta^- x}|\psi_0| \quad (x \le 0).
  \end{equation}
In the following we show that $s=0$ is a simple eigenvalue of $\Phi(\cdot)$, i.e. $\Phi(0)v_0=0$ and
  $\psi_0^{\top}\Phi'(0)v_0\neq 0$. Assuming the contrary, there exist $\zeta^{\pm}\in \R^2$ such that
  \begin{align*}
   V_{\un}^-(0,0)\zeta^--V_{\st}^+(0,0)\zeta^+ & =  \Phi(0) \begin{pmatrix} -\zeta^+ \\ \zeta^- \end{pmatrix} \\
   & =\Phi'(0) v_0=
     \partial_sV_{\st}^+(0,0)v^+- \partial_sV_{\un}^-(0,0) v^-.
  \end{align*}

  With these data the function defined by
  \begin{align*}
    z(x) = \begin{cases} \partial_s V_{\st}^+(0,x)v^+  +V_{\st}^+(0,x) \zeta^+, & x \ge 0, \\
       \partial_sV_{\un}^-(0,x) v^- + V_{\un}^-(0,x) \zeta^-, & x< 0 \end{cases}
  \end{align*}
  is continuous at $x=0$. Moreover, it satisfies the inhomogenous variational equation obtained by differentiating
  $\M(s)V(s,\cdot)=0$ for $V=V_{\st}^+,V_{\un}^-$ w.r.t. $s$ at $s=0$:
    \begin{align} \label{eq3:inhomvar}
      (\partial_x-M(0,\cdot))z =\partial_sM(0,\cdot)\begin{cases} V_{\st}^+(0,\cdot)v^+ ,& x \ge 0 \\
        V_{\un}^-(0,\cdot) v^-, & x<0 \end{cases}
      =
      \begin{pmatrix} 0& 0 \\ A^{-1} & 0 \end{pmatrix} \begin{pmatrix} v_{\star,x} \\ v_{\star,xx} \end{pmatrix} .
    \end{align}
    Rewriting this for the first part $z_1(\cdot)\in \R^2$  of $z=(z_1,z_2)^{\top}$ shows $z_1 \in H^2_{\mathrm{loc}}(\R) $
    and $Lz_1 = v_{\star,x}$. Finally, the exponential estimates \eqref{eq3:sderiv} and \eqref{eq3:sderiv2} imply exponential decay of
    the $s$-derivatives $\partial_s V_{\st}^+(0,\cdot), \partial_sV_{\un}^-(0,\cdot)$ and then of $z$ and its first derivative
    $\partial_xz$ due to \eqref{eq3:inhomvar}.
    Thus we find that $z_1 \in H^2(\R)$ is a generalized eigenfunction of $L$ that belongs to the eigenvalue $s=0$. This
    contradicts Assumption \ref{A4}.
    
      Now we can normalize  $\psi_0$ such that 
    $\psi_0^{\top} \Phi'(0) v_0 = 1$ holds. Then Keldysh's Theorem (see e.g. \cite[Theorem 2.4]{Beyn12}) gives a representation of the inverse $\Phi(s)^{-1}$ near $s=0$. There exists $\delta >0$ and an analytic function $\Gamma: B_\delta(0) \rightarrow \C^{4,4}$
    such that
\begin{align} \label{newPhi-1}
  \Phi(s)^{-1} = \frac{1}{s} v_0 \psi_0^{\top} + \Gamma(s) \quad \forall s \in B_\delta(0) \setminus \{ 0 \}.
\end{align}
Using \eqref{eq3:inhomvar} and integration by parts the normalization is rewritten as follows:
  \begin{align*}
    & -(\psi,\partial_s M(0,\cdot)\varphi)_{L^2} \\
    & =-(\psi,\partial_s M(0,\cdot)V_{\st}^+(0,\cdot)v^+)_{L^2(\R_+)}-
    (\psi,\partial_s M(0,\cdot)V_{\un}^-(0,\cdot)v^-)_{L^2(\R_-)}\\
    & = - (\psi,(\partial_x-M(0,\cdot))\partial_sV_{\st}^+(0,\cdot)v^+ ) _{L^2(\R_+)} \\
    &  \quad - (\psi,(\partial_x-M(0,\cdot))\partial_sV_{\un}^-(0,\cdot)v^- ) _{L^2(\R_-)}\\
    & = \psi(0)^{\top}\partial_sV_{\st}^+(0,0)v^+ - \psi(0)^{\top}\partial_sV_{\un}^-(0,0)v^- = \psi_0^{\top}\Phi'(0)v_0=1.
  \end{align*}
  
    {\bf Step 3: Construction of the solution}\\
      Let us define for $s\in B_{\delta}(0)$ two Green's functions via
\begin{align*}
	\G^-(s,x,y)& = \begin{cases} - \S(s,x,y) P_\un^-(s,y), & x < y \le 0, \\ 
	\S(s,x,y) P_\st^-(s,y), & y \le x \le 0, \end{cases} \\
	\G^+(s,x,y) &= \begin{cases} \S(s,x,y) P_\st^+(s,y), & 0 \le y \le x, \\ 
	- \S(s,x,y)P_{\c \un}^+(s,y) , & 0 \le x < y. \end{cases}
\end{align*}
So far, we have worked with $s \in B_{\delta}(0)$, but  from now on we restrict to $s\in \Omega_{\delta}\cup \{0\}$.
First consider particular solutions of the inhomogeneous equation \eqref{firstvar} on $\R_{\pm}$
  \begin{equation} \label{eq3:modz}
  \begin{aligned}
    z_+(s,x)& = V_{\st}^+(s,x) \zeta_+(s) + \int_{\R_+} \G^+(s,x,y) R(y) dy, \\
    z_-(s,x)& = -V_{\un}^+(s,x) \zeta_-(s) + \int_{\R_-} \G^-(s,x,y) R(y) dy,
  \end{aligned}
  \end{equation}
  where $V_{\st}^+,V_{\un}^-$ are given by \eqref{eq3:Vsol} and $\zeta_{\pm}(s) \in \C^2$ are determined to ensure continuity
  of $z(s,\cdot)=z_+(s,\cdot) \one_{\R_+}+ z_-(s,\cdot)\one_{\R_-}$ at $x=0$, i.e.
  \begin{align*}
    \Phi(s) \begin{pmatrix} \zeta_+(s) \\  \zeta_-(s) \end{pmatrix}= \gamma(s):= \int_{\R_-} \G^-(s,0,y) R(y) dy -\int_{\R_+} \G^+(s,0,y) R(y) dy.
  \end{align*}

  With \eqref{newPhi-1} we  write the solution for $s \in \Omega_{\delta}$ as
  \begin{equation} \label{eq3:zetarep}
    \begin{pmatrix} \zeta_+(s) \\  \zeta_-(s) \end{pmatrix}= \frac{1}{s}v_0 \psi_0^{\top}\gamma(0)+ \gamma_1(s), \quad
    \gamma_1(s)= \Gamma(s)\gamma(s) + v_0 \psi_0^{\top}\int_0^1 \gamma'(\tau s) d\tau.
  \end{equation}
  For the principal term we use $\psi_0^{\top}P_{\un}^-(0,0)=0=\psi_0^{\top}P_{\st}^+(0,0)$ and the definition of $\psi$:
  \begin{align*}
    \psi_0^{\top}\gamma(0) & = \psi_0^{\top} \Big( \int_{\R_-}S(0,0,y)P_{\st}^-(0,y)R(y)dy + \int_{\R_+}S(0,0,y)P_{\c \un}^+(0,y) R(y) dy\Big) \\
    & = \int_{\R_-}\psi_0^{\top}(P_{\st}^-(0,0)+ P_{\un}^-(0,0))S(0,0,y) R(y) dy\\
    & +\int_{\R_+}\psi_0^{\top}(P_{\st}^+(0,0)+ P_{\c \un}^+(0,0))S(0,0,y) R(y) dy\\
    & = \int_{\R_-}\psi(y)^{\top} R(y)dy + \int_{\R_+}\psi(y)^{\top}R(y)dy = ( \psi,R)_{L^2}.
  \end{align*}
  Inserting this into \eqref{eq3:modz} and using $\varphi(x)=S(0,x,0)\varphi_0= \begin{pmatrix}v_{\star,x} & v_{\star,xx} \end{pmatrix}^{\top}$ leads to the following expression for $x \in \R_+$
  \begin{equation} \label{eq3:z+express}
  \begin{aligned}
    z_+(s,x)& =\frac{1}{s} \varphi(x) ( \psi, R)_{L^2}+(\psi,R)_{L^2}  \int_0^1 \partial_s V_{\st}^+(\tau s,x)v^+ d\tau + V_{\st}^+(s,x) \gamma_1(s) \\
    & \quad + \int_0^x \S(s,x,y) P_\st^+(s,y) R(y) dy - \int_x^\infty \S(s,x,y) P_{\c\un}^+(s,y)  R(y)dy.
  \end{aligned}
  \end{equation}
  In a similar way we obtain for $x \in \R_-$
  \begin{equation} \label{eq3:z-express}
    \begin{aligned}
    z_-(s,x)& =\frac{1}{s} \varphi(x) ( \psi, R)_{L^2}+(\psi,R)_{L^2}  \int_0^1 \partial_s V_{\un}^-(\tau s,x)v^- d\tau  + V_{\un}^-(s,x) \gamma_1(s) \\
    & \quad + \int_{-\infty}^x \S(s,x,y) P_\st^-(s,y) R(y) dy - \int_x^0 \S(s,x,y) P_\un^-(s,y) R(y) dy.
    \end{aligned}
    \end{equation}
    Next we estimate $Q_k R = \varphi (\psi,R)_{L^2}$ and the expressions \eqref{eq3:z+express}, \eqref{eq3:z-express}. 
    
{\bf Step 4: Estimates for $s \in \Omega_{\delta}$} \\
We show \eqref{Yestnew2} and \eqref{Yestnew} for $z(s,\cdot)$ on $\R_+$ by estimating term by term in
\eqref{eq3:z+express}. The estimates of $z(s,\cdot)$ on $\R_-$ via  the terms in \eqref{eq3:z-express} are
somewhat simpler (because of the dichotomies) and are therefore omitted here. \\
For the first term in \eqref{eq3:z+express} we note that  $\|\varphi\|_{L^2_p} < \infty$ holds for all $p \ge 1$ due to the
exponential decay and $\|\eta^{-k}\psi\|_{L^2}< \infty$ due to \eqref{eq3:psiest}. This leads to 
  \begin{align} \label{eq3:psiR}
    \|Q_kR\|_{L^2_p}= \|\varphi\|_{L^2_p}|(\eta^{-k} \psi, \eta^k R)_{L^2}|\le \|\varphi\|_{L^2_p}\|\eta^{-k} \psi\|_{L^2}\| R\|_{L^2_k}
    \le C \|R\|_{L^2_k}.
  \end{align}
For the second term we have from \eqref{eq3:sderiv}  that $|\partial_sV_{\st}^+(s,x)| \le C e^{\alpha^+ x}$ holds for $x\ge 0$ and $|s| \le \delta$ yields
\begin{align*}
\begin{split}
	\int_0^\infty \eta^{2k}(x) & \Big| (\psi, R)_{L^2} \int_0^1 \partial_s V_\st^+(\tau s,x) v^+ d\tau \Big|^2 dx \\
	& \qquad \qquad \le C \int_0^\infty \eta^{2k}(x) e^{2 \alpha^+ x} dx \| R \|_{L^2_k}^2 \le C \| R \|_{L^2_k}^2.
\end{split}
\end{align*}
For the third term in \eqref{eq3:z+express} we first estimate $|\gamma_1(s)|$ as follows:
\begin{align*}
  |\gamma_1(s)| & \le C (|\gamma(s)| + \sup_{|\zeta|\le |s|}|\gamma'(\zeta)|) \\
  & \le \sup_{|\zeta|\le |s|}\int_{\R_-} \big(| \G^-| +|\partial_s\G^-|\big)(\zeta,0,y)|R(y)| dy\\
  & + \sup_{|\zeta|\le |s|}\int_{\R_+} \big( |\G^+|+|\partial_s\G^+|\big)(\zeta,0,y)|R(y)| dy \\
  & \le \int_\R \eta^{-2k}(y) dy \| R \|_{L^2_k} \le C \| R \|_{L^2_k},
\end{align*}
where we used that $\G^\pm(\zeta,0,y)$ and $\partial_s \G^\pm(\zeta,0,y)$ are uniformly bounded for $y \in \R_\pm$ and $|\zeta| \le \delta$; see \eqref{eq3:sderiv}, \eqref{eq3:sderiv2}. This gives for the third term in \eqref{eq3:z+express} the estimate
\begin{align} \label{3rdterm}
	\int_0^\infty \eta^{2k}(x) |V_\st^+(s,x) \gamma_1(s)|^2 dx \le C \int_0^\infty \eta^{2k}(x) e^{2 \alpha^+ x} |\gamma_1(s)|^2 dx \le C \| R \|_{L^2_k}^2. 
\end{align}
The fourth term in \eqref{eq3:z+express} involves the stable projector $\P_\st^+(s,\cdot)$. We use \eqref{eq3:est3}, the
Cauchy-Schwarz inequality, and Fubini's theorem, to obtain
\begin{align*}
 & \int_0^\infty \eta^{2k}(x) \Big| \int_0^x \S(s,x,y) P_\st^+(s,y) R(y) dy \Big|^2 dx \\
 & \le C \int_0^\infty \eta^{2k}(x) \Big( \int_0^x e^{\alpha^+ (x-y)} |R(y)| dy \Big)^2 dx \\
 & \le C \int_0^\infty \eta^{2k}(x) \Big( \int_0^x  \eta^{-2k}(y) e^{\alpha^+ (x-y)} dy \Big) \Big( \int_0^x e^{\alpha^+(x-y)} \eta^{2k}(y)|R(y)|^2 dy \Big) dx \\
 & \le C|\alpha^+|^{-2k} \int_0^\infty \int_0^x e^{\alpha^+(x-y)} \eta^{2k}(y) |R(y)|^2 dy dx \\
 & = C|\alpha^+|^{-2k} \int_0^\infty \Big( \int_y^\infty e^{\alpha^+(x-y)} dx \Big) \eta^{2k}(y) |R(y)|^2 dy \\
 & \le C|\alpha^+|^{-2k-1} \int_0^\infty \eta^{2k}(y) |R(y)|^2 dy \le C \| R \|_{L^2_k}^2.
\end{align*}
Next we estimate the last term in \eqref{eq3:z+express} which is the critical term since it involves the central projector through $P_{\c\un}^+(s,\cdot) = P_{\c}^+(s,\cdot) + P_{\un}^+(s,\cdot)$. We decompose it into
\begin{align*}
	\int_x^\infty \S(s,x,y) P_{\c\un}^+(s,y) R(y) dy = & \int_x^\infty \S(s,x,y) P_{\c\un}^+(s,y) Q_k R(y) dy \\
	& + \int_x^\infty \S(s,x,y) P_{\c\un}^+(s,y) (I - Q_k) R(y) dy.
\end{align*}
Since $Q_k R \in L^2_{p}$ for all $p \ge 1$ we obtain by using \eqref{eq3:est1} with $q = 1$ and $2k + 1 +2\mu$ instead of $k$, Cauchy-Schwarz inequality, and Fubini's theorem
\begin{equation} \label{uncriticalterm}
\begin{aligned}
 & \int_0^\infty \eta^{2k}(x) \Big| \int_x^\infty \S(s,x,y) P_{\c\un}^+(s,y) Q_k R(y) dy \Big|^2 dx \\
 & \le C \int_0^\infty \eta^{2k}(x) \Big( \int_x^\infty e^{\nu(s)(x-y)} |Q_k R(y)| dy \Big)^2 dx \\
 & \le C \int_0^\infty \eta^{2k}(x) \Big( \int_x^\infty  \eta^{-2(k + 1 + \mu)}(y) e^{2 \nu(s) (x-y)} dy \Big) \\
 & \qquad \qquad \qquad \qquad \qquad \Big( \int_x^\infty \eta^{2(k+1+\mu)}(y)|Q_k R(y)|^2 dy \Big) dx \\
 & \le C \int_0^\infty \eta^{-1-2\mu}(x) \int_x^\infty \eta^{2(k+1+\mu)}(y) |Q_k R(y)|^2 dy dx \\
 & \le C \int_0^\infty \Big( \int_0^y \eta^{-1-2\mu}(x) dx \Big) \eta^{2(k+1+\mu)}(y) |Q_k R(y)|^2 dy \\
 & \le C \| Q_k R \|_{L^2_{k+1+\mu}}^2 \le C \| Q_k R \|_{L^2_{k}}^2.
\end{aligned}
\end{equation}
If $R \in L^2_{k+1+\mu}$ we can replace $Q_k R$ by $(I-Q_k) R$ in \eqref{uncriticalterm} and obtain
\begin{equation} \label{criticalestimate}
\begin{aligned}
	\int_0^\infty \eta^{2k}(x) \Big| \int_x^\infty \S(s,x,y) P_{\c\un}^+(s,y) (I - Q_k) & R(y) dy \Big|^2 dx \\
	& \le C \| (I - Q_k) R \|_{L^2_{k+1+\mu}}^2.
\end{aligned}
\end{equation}
It is here that we need the stronger norm $\| (I - Q_k) R \|_{L^2_{k + 1 + \mu}}$. Collecting the estimates above  we infer \eqref{Yestnew} for $z(s,\cdot)$ on $\R_+$ in
case $R \in L^2_{k+1+\mu}$.
If we have only $R \in L^2_{k+q}$ with $0\le q <1$ then  we modify the estimate of the critical term from \eqref{criticalestimate} by using \eqref{eq3:est1}, \eqref{eq3:est4}, and \eqref{trichoplus}:
\begin{align*}
 & \int_0^\infty \eta^{2k}(x) \Big| \int_x^\infty \S(s,x,y) P_{\c\un}^+(s,y) (I - Q_k) R(y) dy \Big|^2 dx \\
 & \le C \int_0^\infty \eta^{2k}(x) \Big( \int_x^\infty e^{\nu(s)(x-y)} |(I - Q_k) R(y)| dy \Big)^2 dx \\
 & \le C \int_0^\infty \eta^{2k}(x) \Big( \int_x^\infty  \eta^{-2(k + q)}(y) e^{\nu(s) (x-y)} dy\Big) \\
 & \qquad \qquad \qquad \qquad \qquad \Big( \int_x^\infty e^{\nu(s)(x-y)} \eta^{2(k+q)}(y)|(I - Q_k) R(y)|^2 dy \Big) dx \\
 & \le C|\nu(s)|^{q-1} \int_0^\infty \eta^{-q}(x) \int_x^\infty e^{\nu(s)(x-y)} \eta^{2(k+q)}(y) |(I - Q_k) R(y)|^2 dy dx \\
 & \le C|\nu(s)|^{q-1} \int_0^\infty \Big( \int_0^y \eta^{-q}(x) e^{\nu(s)(x-y)} dx\Big) \eta^{2(k+q)}(y) |(I - Q_k) R(y)|^2 dy \\
 & \le C|\nu(s)|^{2q-2} \int_0^\infty \eta^{2(k+q)}(y) |(I - Q_k) R(y)|^2 dy \le C |s|^{4q - 4} \| (I - Q_k) R \|_{L^2_{k+q}}^2,
\end{align*}
where the last inequality uses the estimate $|s|^2 \le C |\lambda(s)|^2 \le C \nu(s)$ from Lemma \ref{dichomatrices} (ii) b), c) for $s \in \Omega_\delta$. Together with the estimates above this proves \eqref{Yestnew2} for $z(s,\cdot)$ on $\R_+$ in case $R \in L^2_{k+q}$. Combining this  with the corresponding estimates on $\R_-$ proves our assertions \eqref{Yestnew2} and \eqref{Yestnew} for $s \in \Omega_{\delta}$.

{\bf Step 5: The case $s=0$} \\
Let us reconsider steps 3 \& 4 in case $(\psi,R)_{L^2}=0$ with $R \in L^2_{k+1+\mu}$ and let $s=0$. Then we have $\psi_0^{\top}\gamma(0)=0$ and
\eqref{eq3:zetarep} shows that $\begin{pmatrix} \zeta_+(s) \\ \zeta_-(s) \end{pmatrix}=\gamma_1(s)$
is continuous at $s=0$. Moreover, the term $(\psi,R)_{L^2}$ drops out from the definition of $z_{\pm}$ in
\eqref{eq3:z+express}, \eqref{eq3:z-express} and the subsequent estimates \eqref{3rdterm}-\eqref{criticalestimate} still hold for $s = 0$ and $\nu(0)=0$. Continuity of the integral expressions \eqref{eq3:z+express}, \eqref{eq3:z-express} at $s=0$
follows from the uniform bounds \eqref{Yestnew} and the pointwise convergence of the integrands via Lebesgue's theorem.
\end{proof}

As a consequence of Proposition  \ref{propEV} we obtain an estimate for the solution of the resolvent equation \eqref{poly:resolvGl}  in 
 a sufficiently small $\delta$-crescent $\Omega_\delta$.

\begin{corollary} \label{cor213}
  Let Assumption \ref{A1}-\ref{A4} be satisfied and let $k \ge 1$. Then there exist $\delta > 0$ and for every $0 \le q < 1$, $\mu>0$ a constant $C=C(k,q,\mu)$ such that the equation $(sI-L)v = r$ has a unique solution $v(s) \in H^2_{k}$ for all $s \in \Omega_\delta $ and $r \in L^2_{k}$. This solution satisfies the estimates
\begin{alignat}{2}
\| v(s) \|_{H^1_k} & \le C\big(\frac{1}{|s|}\|P_kr\|_{L^2_{k}}+ |s|^{2q -2}\|(I-P_k) r \|_{L^2_{k+q}}\big) \quad && \text{if } r \in L^2_{k+q}. \label{eq3:vsingest2} \\
	\| v(s) \|_{H^1_k} & \le C\big(\frac{1}{|s|}\|P_kr\|_{L^2_{k}}+ \|(I-P_k) r \|_{L^2_{k+1+\mu}}\big) \quad && \text{if } r \in L^2_{k+1+\mu}. \label{eq3:vsingest}
\end{alignat}
Here $P_k: L_{k}^2 \to \ker(L)$ is the projector given by
\begin{align*}
  P_k r= v_{\star,x}(\psi_2,r)_{L^2}
\end{align*}
and $\psi_2$ is the unique bounded solution of the adjoint system
\begin{align} \label{eq3:adeq}
  L^{\star}\psi_2= (A^{\top}\partial_x^2 - c \partial_x + S_{\omega}^{\top}+Df(v_{\star})^{\top})\psi_2 = 0, \quad
  (\psi_2,v_{\star,x})_{L^2} = 1.
\end{align}
If $r \in L^2_{k+1+\mu}$ with $P_kr=0$ then $v(s)$, $s \in \Omega_{\delta}$ can be continuously extended by some $v(0)\in H_{k}^2$ which solves $-Lv(0)=r$ and which
satisfies the estimate $\|v(0)\|_{H^1_k} \le C \|r\|_{L^2_{k+1+\mu}}$.
\end{corollary}

\begin{proof}
  By a  straightforward calculation we find that the equation
  \begin{align*}
    0& = (\partial_x + M(0,\cdot)^{\top})
    \begin{pmatrix} \psi_1 \\ A^{\top}\psi_2 \end{pmatrix}=
    \begin{pmatrix} \partial_x \psi_1 - (S_{\omega}^{\top}+Df(v_{\star})^{\top})\psi_2 \\
      A^{\top} \partial_x \psi_2 + \psi_1 - c \psi_2 \end{pmatrix}
  \end{align*}
  holds if and only if $L^{\star}\psi_2=0$. Setting $\psi=-\begin{pmatrix} \psi_1 \\ A^{\top}\psi_2 \end{pmatrix}$  and
  taking \eqref{firstvar} into account yields
  \begin{align*}
    Q_k R =\begin{pmatrix}v_{\star,x} \\ v_{\star,xx} \end{pmatrix}\Big( \begin{pmatrix}- \psi_1\\ -A^{\top}\psi_2 \end{pmatrix}, \begin{pmatrix} 0 \\- A^{-1} r \end{pmatrix}\Big)_{L^2}
    =(\psi_2,r)_{L_2}\begin{pmatrix}v_{\star,x} \\ v_{\star,xx} \end{pmatrix},
  \end{align*}
  where $P_kr$ appears in the first component. By \eqref{eq3:adjoint} the normalization is 
  \begin{align*}
    1 = - (\psi,\partial_sM(0,\cdot)\varphi)_{L^2}= (A^{\top}\psi_2, A^{-1}v_{\star,x})_{L^2} = (\psi_2,v_{\star,x})_{L^2}.
  \end{align*}
  Hence, Proposition \ref{propEV} shows that \eqref{eq3:adeq} has a unique 
  bounded solution. Further,
  the setting $z(s)=\left( \begin{smallmatrix} v(s)\\ \partial_xv(s) \end{smallmatrix} \right)$
  transfers the estimate \eqref{Yestnew} to our assertion \eqref{eq3:vsingest}. This one uses the estimate of
  the projector $P_k$ which is analogus to \eqref{eq3:psiR}:
  \begin{align*}
    \|P_k r \|_{L^p}\le C(p,k) \|r \|_{L^2_k} \quad \forall r \in L^2_k, k, p \ge 1.
    \end{align*}
    This completes the proof.
\end{proof}

Corollary \ref{cor213} suggests to introduce the following subspaces for $k,\ell \in \N$
\begin{align*}
	X_{k} := (I - P_k)(L^2_k), \quad X^\ell_k := X_k \cap H^\ell.
\end{align*}
This leads to the decompositions $L^2_k = \ker(L) \oplus X_k$ and $H^2_k = \ker(L) \oplus X^2_k$. Further, $L(X^2_k) \subset X_k$ holds since for $v \in X^2_k$ we have $Lv \in L^2_k$ and
\begin{align*}
	P_k L(I - P_k) v = P_k Lv = (\psi_2, Lv)_{L^2} v_{\star,x} = (L^* \psi_2, v)_{L^2} v_{\star,x} = 0.
\end{align*}

\begin{remark} \label{LFred}
In fact, one can show that the operator $L:\D(L) \subset L^2_k \rightarrow L^2_{k+1+\mu}$ with domain $\D(L) = \{ u \in H^1_k \cap H^2_{\mathrm{loc}}: L u \in L^2_{k+1+\mu} \}$ is a closed operator from $L^2_k$ to $L^2_{k+1+\mu}$ and Fredholm of index $0$; cf. \cite[Sec. 5.3]{Doeding}. Taking the subspaces $X_k$ into account then implies that the operator $L: \D(L) \cap X_k \subset X_k \rightarrow X_{k+1+\mu}$ is invertible.
\end{remark}

\sect{Semigroup estimates}

In this section we use the estimates from Corollary \ref{cor213} to construct and prove algebraic decay of the semigroup generated by $L$. From Lemma \ref{poly:aprioriest} and Theorem \ref{Thm:essential} we conclude that the operator $L$ from \eqref{LatL} lies in
$\clos[L^2_k]$, is a sectorial operator, and generates an analytic semigroup denoted by $\{ e^{tL} \}_{t \ge 0} $. Moreover,
Assumption \ref{A4} guarantees the spectral
bound at zero, hence  there exists $\Cexp \ge 1$ such that for $\ell = 0,1$ the following estimates hold
\begin{align} \label{semiexpest}
	\| e^{tL} u \|_{H^\ell_k} \le \Cexp e^{t} \| u \|_{H^\ell_k}, \quad \| e^{tL} u \|_{H^1_k} \le \Cexp t^{-\frac{1}{2}}e^{t}\| u \|_{L^2_k}.
\end{align}
Note that \eqref{semiexpest} is not sharp for large $t \ge 0$ in the sense that in fact we have the bounds $\| e^{tL} u \|_{H^\ell_k} \le C(\beta) e^{\beta t} \| u \|_{H^\ell_k}$ and $\| e^{tL} u \|_{H^1_k} \le C(\beta) t^{-\frac{1}{2}}e^{\beta t}\| u \|_{L^2_k}$ for arbitrary $\beta > 0$. However, for simplicity we chose $\beta = 1$ so that $\Cexp = C(1)$ is fixed. \\
The essential spectrum of $L$ touching  the imaginary axis at the origin prevents the semigroup to have exponential decay.
Instead, we use the sharp resolvent estimates  of the previous sections to show  algebraic decay of the semigroup for
initial data from a proper subspace.

\begin{theorem} \label{poly4:semigroup}
Let Assumption \ref{A1}-\ref{A4} be satisfied and $k, m \ge 1$ and $\mu > 0$. Then there exists $\Calg = \Calg (k, m,\mu) \ge 1$ such that the linearized operator $L: H^2_k \rightarrow L^2_k$ generates an analytic semigroup $\{ e^{t L}\}_{t \ge 0}$ on $L^2_k$ that satisfies for all $v \in X_{k+m+\mu}$ and $\ell = 0,1$ the estimates
\begin{align}
	 & \| e^{t L} v \|_{H^\ell_k} \le \Calg (1+t)^{-\frac{m^*}{2}} \| v \|_{H^\ell_{k+m+\mu}}, \quad \| e^{tL} v \|_{H^1_k} \le \Calg t^{-\frac{m^*}{2}} \| v \|_{L^2_{k+m+\mu}} \label{poly:semi}
\end{align}
where $m^* = \lfloor m \rfloor + \max(0,2q-1)$ for $m = \lfloor m \rfloor + q$, $0 \le q < 1$.
\end{theorem}

\begin{proof} 
We first prove \eqref{poly:semi} for the case $m = 1$: Define oriented contours $\Gamma_\pm$ running upwards and $\Gamma_\delta$ running counterclockwise in the complex plane (see Figure \ref{figure-contours} (left)).
\begin{align*}
	& \Gamma_+ := \{ z + \tau e^{i(\frac{\pi}{2} + \theta_0)}: \tau \ge 0 \}, \\
	& \Gamma_- := \{ \overline{z} + \tau e^{-i(\frac{\pi}{2} + \theta_0)}:  \tau \le 0 \}, \\
	& \Gamma_\delta := \{ s \in \Omega_\delta: |s| = \delta \}
\end{align*}
where $z := \mathrm{argmin} \{ \Re s: s \in \Omega_\delta, \, \Im s \ge 0 \}$ and $\theta_0 < \varepsilon_0$ with $\varepsilon_0$ from Lemma \ref{poly:aprioriest}. If we choose $\delta$ sufficiently small we have
\begin{align*}
	z = - \kappa R^2 + i R, \quad R^2 = - \frac{1}{2\kappa^2} + \sqrt{ \frac{1}{4\kappa^4} + \frac{\delta^2}{\kappa^2} }.
\end{align*}
By further taking $\delta$ and $\theta_0$ sufficiently small, we can ensure, by Assumption \ref{A4} and Theorem \ref{Thm:essential}, that the concatenated contour $\Gamma := \Gamma_- \cup \Gamma_\delta \cup \Gamma_+$ satisfies $\Gamma \subset \mathrm{res}(L)$ for $L \in \clos[L^2_k]$ and that there is no spectrum of $L$ to the right of $\Gamma$ in the complex plane.

\begin{figure}[h!]
\centering
\begin{minipage}[t]{0.45\textwidth}
\centering
\includegraphics[scale=0.35]{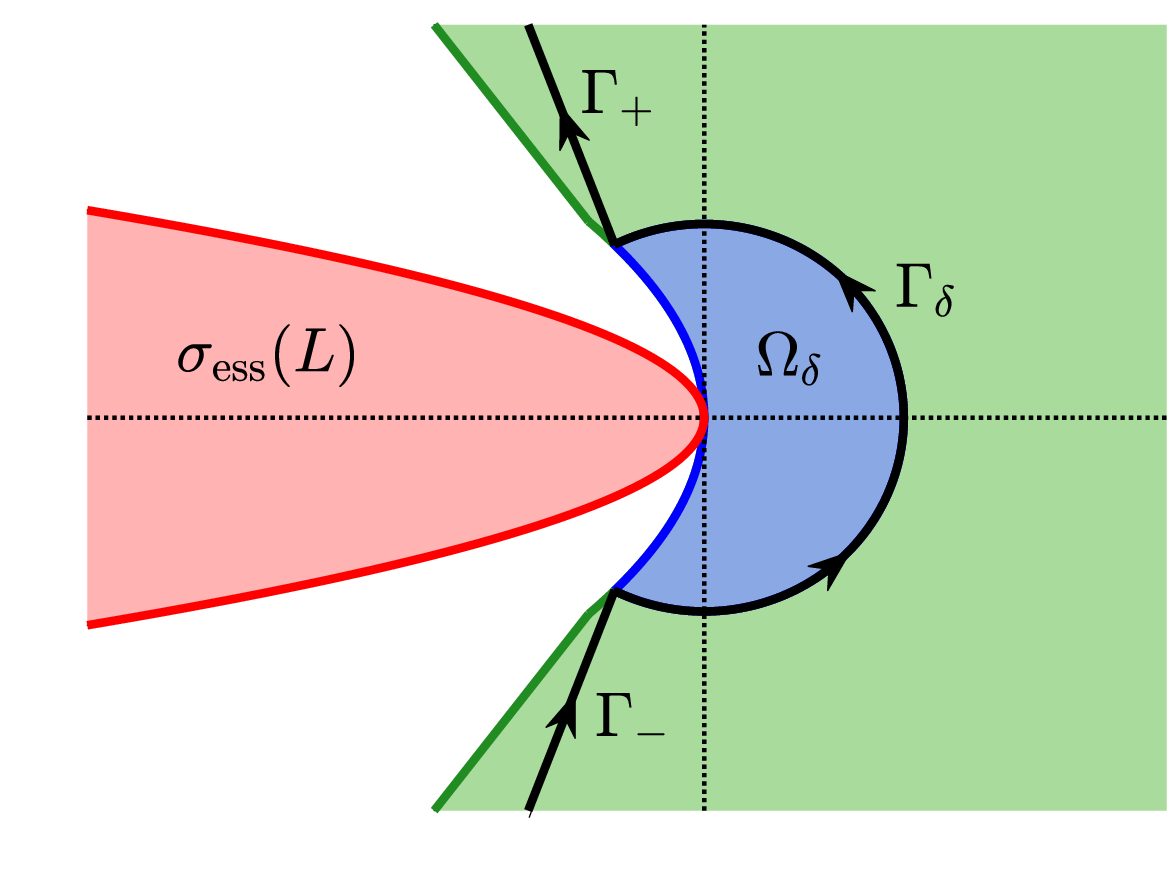}
\end{minipage}
\begin{minipage}[t]{0.45\textwidth}
\centering
\includegraphics[scale=0.35]{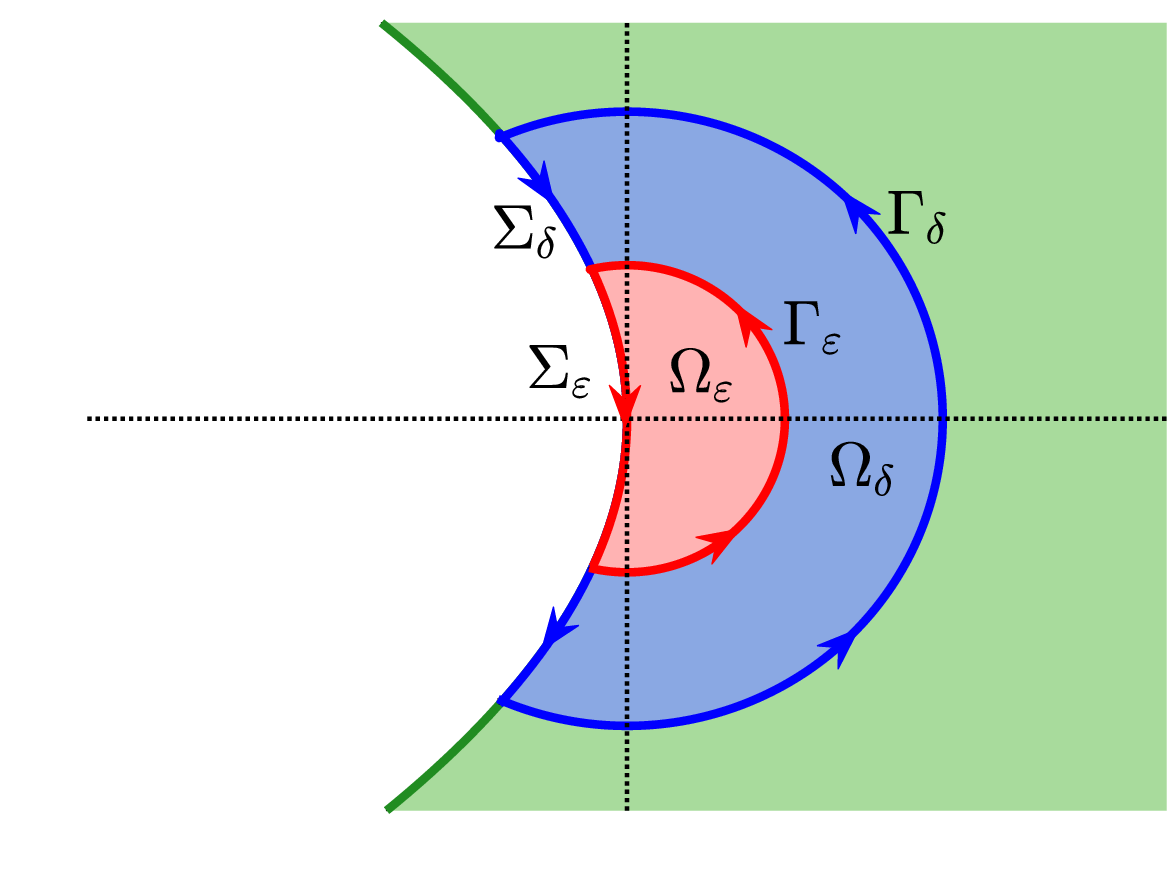}
\end{minipage}
\caption{Contours used in the proof of Theorem \ref{poly4:semigroup}.} \label{figure-contours}
\end{figure}

Using the sectorial estimates from Lemma \ref{poly:aprioriest},  standard semigroup theory (see e.g. \cite{Henry,RenardyRogers,Pazy}) implies that
\begin{align*}
	e^{tL} := \frac{1}{2\pi i} \int_\Gamma e^{ts} (sI - L)^{-1} ds
\end{align*}
defines an analytic semigroup on $L^2_k$ satisfying  \eqref{semiexpest} for any $\beta > 0$ and all $v \in H^\ell_k$, $\ell = 0,1$. The boundary of the closure $\overline{\Omega}_{\delta}= \Omega_{\delta} \cup \{0\}$ is given by the contour
$\partial \Omega_{\delta}= \Sigma_{\delta} \cup \Gamma_{\delta}$ with the parabolic curve (see Figure \ref{figure-contours} (right))
\begin{align*}
	\Sigma_\delta := \{ -\kappa \tau^2 - i \tau : - R \le \tau \le R \}.
\end{align*}
If $v \in X_{k+1+\mu}$ holds, then by Corollary \ref{cor213} 
the function $z(s)=(sI-L)^{-1}v $, $s \in \Omega_{\delta}$ has
a continuous extension to  $\overline{\Omega}_{\delta}$ which we write as $z(0)= (0-L)^{-1}v$. In this sense  the integral
\begin{align} \label{OmegaIntegral}
	\int_{\partial \Omega_\delta} e^{ts} (sI - L)^{-1} v ds
\end{align}
exists in $L^2_k$. We show that it vanishes although $\overline{\Omega}_{\delta}$ is not contained in an open
domain of analyticity of the resolvent \footnote{Note that $0$ is not a removable singularity in the classical
sense which requires analyticity to hold in a pointed neighborhood.}. For that purpose, let $\Omega_\varepsilon$ be the $\varepsilon$-crescent according to \eqref{crescent} for $0 < \varepsilon \le \delta$. Since the resolvent $(sI - L)^{-1}$ is analytic in $\mathrm{res}(L)\supset \Omega_{\delta}\setminus \Omega_{\varepsilon}$ an application of Cauchy's integral theorem and of the uniform estimates from Corollary \ref{cor213} imply
\begin{equation} \label{contourlimit}
	\left\| \int_{\partial \Omega_\delta} e^{ts} (sI - L)^{-1} v ds \right\|_{L^2_k} = \left\|  \int_{\partial \Omega_\varepsilon} e^{ts} (sI - L)^{-1} v ds \right\|_{L^2_k}  \le C \varepsilon \|v\|_{L^2_{k+1+\mu}} \rightarrow 0
\end{equation}
as $\varepsilon \rightarrow 0$. Using Corollary \ref{cor213} and \eqref{contourlimit} this gives for $t \ge 1$ and for a generic constant $C > 0$
\begin{align*}
	\left\| \int_{\Gamma_\delta} e^{ts} (sI - L)^{-1} v ds \right\|_{H^1_k} & = \left\| \int_{\Sigma_\delta} e^{ts} (sI - L)^{-1} v ds \right\|_{H^1_k} \\
	& \le C \| v \|_{L^2_{k+1+\mu}} \int_{-R}^R e^{-t \kappa \tau^2} |2 \kappa \tau + i| d\tau \\
	& \le  \frac{C}{\sqrt{t}} \| v \|_{L^2_{k+1+\mu}} \int_{0}^{\sqrt{t}R} e^{-\kappa \sigma^2} d\sigma \le \frac{C}{\sqrt{t}} \| v \|_{L^2_{k+1+\mu}}.
\end{align*}
Further, for $t \ge 1$ we have due to $\Gamma_\pm \subset \mathrm{res}(L)$
\begin{align} \label{estimate_Gamma_pm}
\begin{split}
	& \left\| \int_{\Gamma_\pm} e^{ts} (sI - L)^{-1} v ds \right\|_{H^1_k} \le C \| v \|_{L^2_{k}} e^{- \kappa t R^2} \int_0^\infty e^{t \tau \cos(\frac{\pi}{2} + \theta_0)} d\tau \\
	& \le \frac{C}{t} \| v \|_{L^2_{k}} e^{- \kappa t R^2} \int_0^\infty e^{- \sigma |\cos (\frac{\pi}{2} + \theta_0)|} d\sigma \le \frac{C}{t} \| v \|_{L^2_{k}}.
\end{split}
\end{align}
Using the exponential bounds \eqref{semiexpest} for $t \in [0,1]$ we conclude for $t > 0$, $\ell = 0,1$
\begin{align}
	 & \| e^{t L} v \|_{H^\ell_k} \le C(1+t)^{-\frac{1}{2}} \| v \|_{H^\ell_{k+1+\mu}}, \quad \| e^{tL} v \|_{H^1_k} \le C t^{-\frac{1}{2}} \| v \|_{L^2_{k+1+\mu}}. \label{poly:semim1}
\end{align}
If $m$ is an arbitrary integer, i.e. $m \in \N$, we can apply \eqref{poly:semim1} $m$-times to obtain
\begin{align*}
	& \| e^{tL} v \|_{H^\ell_k} = \left\| \left( e^{\frac{t}{m}L} \right)^m v \right\|_{H^\ell_k} \le C(1+\tfrac{1}{m}t)^{-\frac{1}{2}} \left\| \left( e^{\frac{t}{m}L} \right)^{m-1} v \right\|_{H^\ell_{k+1+\mu}} \\
	& \le \dots \le C^m(1+\tfrac{1}{m}t)^{-\frac{m}{2}} \left\| v \right\|_{H^\ell_{k+m(1+\mu)}} \le C (1+t)^{-\frac{m}{2}} \| v \|_{H^\ell_{k+m+m\mu}}.
\end{align*}
Replacing $m \mu$ by $\mu$ yields the first estimate in \eqref{poly:semi} for $m \in \N$. The second one follows analogously from the second estimate in \eqref{poly:semim1}. \\
Let us now consider the case for arbitrary $m \ge 1$ and assume $v \in X_{k + m + \mu}$. We decompose $m = p + q$ with $p \in \N$ and $0 \le q < 1$. If $q \le \frac{1}{2}$, the estimate \eqref{poly:semi} immediately follows by our previous observations and the inclusion $H^\ell_{k + m + \mu} \subset H^\ell_{k + p + \mu}$. So let $\frac{1}{2} < q < 1$. Using the resolvent estimate \eqref{eq3:vsingest2} from Corollary \ref{cor213} we observe that the integral \eqref{OmegaIntegral} exists even in $L^2_{k + p}$ since we can estimate:
\begin{align*}
\begin{split}
	& \left\| \int_{\partial \Omega_\delta} e^{ts} (sI - L)^{-1} v ds \right\|_{L^2_{k+p}} = \left\|  \int_{\partial \Omega_\varepsilon} e^{ts} (sI - L)^{-1} v ds \right\|_{L^2_{k + p}}  \\
	& \le \left\|  \int_{\Gamma_\varepsilon} e^{ts} (sI - L)^{-1} v ds \right\|_{L^2_{k + p}} + \left\|  \int_{\Sigma_\varepsilon} e^{ts} (sI - L)^{-1} v ds \right\|_{L^2_{k + p}} \\
	&\le C \varepsilon^{2q - 1} \|v\|_{L^2_{k+p+q}} \rightarrow 0 \quad \text{ as } \varepsilon \rightarrow 0.
\end{split}
\end{align*}
Further, the estimate \eqref{eq3:vsingest2} from Corollary \ref{cor213} now yields for $t \ge 1$ and a generic constant $C > 0$
\begin{align*}
	& \left\| \int_{\Gamma_\delta} e^{ts} (sI - L)^{-1} v ds \right\|_{H^1_{k+p}} =\left\| \int_{\Sigma_\delta} e^{ts} (sI - L)^{-1} v ds \right\|_{H^1_{k+p}} \\
	& \le C \| v \|_{L^2_{k+p+q}} \int_{-R}^R e^{-t \kappa \tau^2} |2 \kappa \tau + i| |\kappa \tau^2 + i \tau|^{2q - 2} d\tau \\
	& \le C t^{- \frac{1}{2}} \| v \|_{L^2_{k+p+q}} \int_{0}^{\sqrt{t}R} e^{-\kappa \sigma^2} |2 \kappa \sigma t^{- \frac{1}{2}} + i| |\kappa \sigma^2 t^{-1} +i \sigma t^{- \frac{1}{2}} |^{2q - 2} d\sigma \\
	& \le C t^{- \frac{1}{2}} \| v \|_{L^2_{k+p+q}} \int_{0}^{\sqrt{t}R} e^{-\kappa \sigma^2} |2 \kappa \sigma t^{- \frac{1}{2}} + i| (\kappa^2 \sigma^4 t^{-2} + \sigma^2 t^{-1} )^{q - 1} d\sigma \\
	& \le C t^{\frac{1}{2} - q} \| v \|_{L^2_{k+p+q}} (2\kappa R + 1) \int_{0}^{\sqrt{t}R} \sigma^{2q - 2} e^{-\kappa \sigma^2} d\sigma \\
	& \le  C t^{\frac{1}{2} - q} \| v \|_{L^2_{k+p+q}} \int_{0}^{\infty} \sigma^{2q - 2} e^{-\kappa \sigma^2} d\sigma \\
	& \le C t^{-\frac{2q-1}{2}} \| v \|_{L^2_{k+p+q}}.
\end{align*}
With \eqref{estimate_Gamma_pm} and the exponential bounds \eqref{semiexpest} this yields for $t > 0$ and $\ell = 0,1$
\begin{align*}
	 & \| e^{t L} v \|_{H^\ell_{k+p}} \le C(1+t)^{-\frac{2q-1}{2}} \| v \|_{H^\ell_{k+p+q}}, \quad \| e^{tL} v \|_{H^1_{k+p}} \le C t^{-\frac{2q-1}{2}} \| v \|_{L^2_{k+p+q}}.
\end{align*}
We obtain the first estimate in \eqref{poly:semi} as follows:
\begin{align*}
	& \| e^{tL} v \|_{H^1_{k}} = \| e^{\frac{pt}{m}L}  e^{\frac{q t}{m}L} v \|_{H^1_{k}} \le C (1 + \tfrac{p}{m} t)^{-\frac{p}{2}} \| e^{\frac{q t}{m}L} v \|_{H^1_{k + p + \mu}} \\
	& \le C (1 + \tfrac{p}{m}t)^{-\frac{p}{2}} (1 + \tfrac{q}{m}t)^{-\frac{2q-1}{2}} \| v \|_{H^1_{k+ p + q + \mu}} \le C (1 + t)^{-\frac{p + 2q - 1}{2}} \| v \|_{H^1_{k+ m + \mu}}.
\end{align*}
As before, the second estimate in \eqref{poly:semi} for the general case $m \ge 1$ follows by similar arguments. This completes the proof.
\end{proof}

\sect{Decomposition of the dynamics and nonlinearities}
\label{sec:Decomposition}
We decompose the solution of the perturbed co-moving system \eqref{perturbsys} into
\begin{align} \label{decomp}
	v(t) = v_\star( \cdot - \tau(t) ) + w(t)
\end{align}
so that $\tau: [0,\infty) \rightarrow \R$ describes a motion along the group orbit $\mathcal{O}(v_\star) := \{ v_\star(\cdot - \tau),\, \tau \in \R\}$ and $w: [0,\infty) \rightarrow X_k$ describes a perturbation perpendicular to the nullspace of the adjoint operator $L^*$ spanned by $\psi$ from Corollary \ref{cor213}. We follow the lines of \cite{BeynDoeding22,BeynLorenz} to show that the decomposition \eqref{decomp} is unique as long as the solution $v$ stays close to the profile $v_\star$ and further derive a system of evolution equations that determines the dynamics of $\tau$ and $w$. We define the maps
\begin{align} \label{npoly:Pi}
	\Pi_k(\tau) := P_k(v_\star(\cdot - \tau) - v_\star), \quad \tau \in \R
\end{align}
and
\begin{align} \label{poly:T}
	T_k(\tau,w) := v_{\star}(\cdot - \tau) - v_\star + w, \quad \tau \in \R, \, w \in X_k.
\end{align}

\begin{lemma} \label{poly:lemmatrafo}
Let Assumption \ref{A1}-\ref{A4} be satisfied and $k \ge 1$. Then there are zero neighborhoods $\U_\tau \subset \R$ and $\U_w \subset X_k$ such that the maps $\Pi_k: \U_\tau \rightarrow \ker(L)$ from \eqref{npoly:Pi} and $T_k: \U_\tau \times \U_w \rightarrow L^2_k$ from \eqref{poly:T} are local diffeomorphisms. Moreover, the solution of $T_k(\tau,w) = v$ is given by
\begin{align} \label{trafo:identity}
	\tau = \Pi_k^{-1}(P_kv), \quad w = v + v_\star - v_\star(\cdot - \tau)
\end{align}
if $(\tau,w) \in \U_\tau \times \U_w$.
\end{lemma}

\begin{proof}
First note that for every $\tau \in \R$ the function $v_{\star}(\cdot  - \tau) - v_\star$ is continuous in $X_k$ and decays exponentially at $\pm \infty$ due to Assumptions \ref{A1}, \ref{A2} and \cite[Thm. 2.5]{BeynDoeding22}. Therefore, $\Pi_k$ is well-defined and continuously differentiable due to Corollary \ref{cor213} and satisfies $\Pi_k(0) = 0$ and $\Pi'_k(0) = -v_{\star,x} \neq 0$. Thus the assertion for $\Pi_k$ is a consequence of the inverse function theorem. Similarly, $T_k$ is continuously differentiable and satisfies $T_k(0,0) = 0$ with $T'_k(0,0) = (-v_{\star,x}, \, I_{L^2_k}): \R \times X_k \rightarrow L^2_k$ which is invertible due to the decomposition $L^2_k = \ker(L) \oplus X_k$. Hence, by the inverse function theorem, $T_k$ is a local diffeomorphism and one concludes the identity \eqref{trafo:identity} by applying $P_k$ to the equation $T_k(\tau,w) = v$.
\end{proof}

Assuming that \eqref{perturbsys} has a solution $v(t)$, $t \in [0,t_\infty)$ for some $t_\infty$ such that $v(t) - v_\star \in \U_w$ for all $t \in [0,t_\infty)$, i.e., the solution stays in the region where $T_k^{-1}$ exists, we find unique $\tau: [0,t_\infty) \rightarrow \R$ and $w: [0,t_\infty) \rightarrow X_k$ such that
\begin{align*}
	u(t) - v_\star = T_k(\tau(t),w(t)), \quad t \in [0,t_\infty)
\end{align*}
and the decomposition \eqref{decomp} holds for all $t \in [0,t_\infty)$. In this case, we have $u_0 \in \U_w$ and we can set the initial values for $\tau$ and $w$ as
\begin{align} \label{wtauinitial}
	\tau(0) = \Pi_k^{-1}(P_k u_0) = : \tau_0, \quad w(0) = u_0 + v_\star - v_\star(\cdot - \tau_0) = : w_0.
\end{align}
If $w(t) \in X^2_k$ for $t \in (0,t_\infty)$ we insert \eqref{decomp} into \eqref{perturbsys} and obtain
\begin{align*}
	& w_t - v_{\star,x}(\cdot - \tau) \tau_t = v_t = Lv + f(v) - Df(v_\star)v \\
	& = Lw + Lv_\star(\cdot - \tau) + f(v_\star(\cdot - \tau) + w) - Df(v_\star)v_\star(\cdot -\tau) - Df(v_\star)w \\
	& = Lw + f(v_\star(\cdot - \tau) + w) - f(v_\star(\cdot - \tau)) - Df(v_\star)w
\end{align*}
where we used $L v_\star(\cdot - \tau) - Df(v_\star)v_\star(\cdot - \tau) = f(v_\star(\cdot - \tau))$ for all $\tau \in \R$ by the equivariance of the stationary co-moving equation \eqref{statcomov}. This leads to
\begin{align} \label{w:eq1}
	w_t = Lw + v_{\star,x}(\cdot - \tau)\tau_t + r^{[f]}(\tau,w)
\end{align}
with the nonlinear remainder
\begin{align*}
	r^{[f]}(\tau,w) = f(v_\star(\cdot - \tau) + w) - f(v_\star(\cdot - \tau)) - Df(v_\star)w.
\end{align*}
Recall the projector $P_k$ from Corollary \ref{cor213} and define the continuous linear map
\begin{align*}
	S_k(\tau)=\Pi_k'(\tau) \colon \R \rightarrow \ker(L), \quad \mu \mapsto -P_k(v_{\star,x}(\cdot - \tau))\mu.
\end{align*}
Since $(\psi_2,v_{\star,x})_{L^2} = 1$ we can assume w.l.o.g. that $(\psi_2,v_{\star,x}(\cdot - \tau))_{L^2} \neq 0 $ for all $\tau \in \U_\tau$ so that the inverse of $S_k(\tau)$, $\tau \in \U_\tau$ exists and is given by
\begin{align} \label{Sinverse}
	S_k(\tau)^{-1}v := \frac{(\psi_2,v)_{L^2}}{(\psi_2,v_{\star,x}(\cdot - \tau))_{L^2}}, \quad v \in \ker(L).
\end{align}
Note that $S_k(\tau)^{-1}$ is continuously differentiable w.r.t. $\tau \in \U_\tau$. Applying first $P_k$ and then $I - P_k$ to \eqref{w:eq1} and using $P_k(w_t - Lw) = 0$, $L(X^2_k) \subset X_k$ we obtain a system of evolution equations for $\tau$ and $w$ 
\begin{alignat}{2}
	\tau_t & = r^{[\tau]}(\tau,w), && \quad \tau(0) = \tau_0, \label{tau-equation} \\
	w_t & = Lw + r^{[w]}(\tau,w), && \quad w(0) = w_0, \label{w-equation}
\end{alignat}
with $\tau_0,w_0$ from \eqref{wtauinitial} and the nonlinear remainders
\begin{align} 
	r^{[\tau]}(\tau,w) & := S_k(\tau)^{-1}P_k r^{[f]}(\tau,w), \label{tau-remainder}\\
	r^{[w]}(\tau,w) & := (I-P_k)(v_{\star,x}(\cdot - \tau)S_k(\tau)^{-1} P_k + I)r^{[f]}(\tau,w). \label{w-remainder} 
\end{align}
Next we derive Lipschitz estimates for the nonlinearities in \eqref{tau-remainder}, \eqref{w-remainder}.

\begin{lemma} \label{poly:estimates}
Let Assumption \ref{A1}-\ref{A4} be satisfied and $\ell = p + q \ge 0$ with $p,q \ge 0$. Then there is $\delta > 0$ and a constant $C = C(\delta) > 0$ such that for all $\tau,\tau_1,\tau_2 \in B_{\delta}(0) \subset \R$ and $w,w_1,w_2 \in B_{\delta}(0) \subset H^1_{p} \cap H^1_{q}$ there hold:
\begin{align}
	& \left\| r^{[f]}(\tau, w_1) - r^{[f]}(\tau, w_2) \right\|_{H^1_{\ell}} \le C \left( |\tau| + \max ( \| w_1 \|_{H^1_{p}}, \| w_2 \|_{H^1_{p}} ) \right) \| w_1 - w_2 \|_{H^1_q} \label{nonlinear-estimate-1} \\
	&  \left\| r^{[f]}(\tau_1, w) - r^{[f]}(\tau_2, w) \right\|_{H^1_{\ell}} \le C|\tau_1 - \tau_2|, \label{nonlinear-estimate-2} \\
	& \left\| r^{[w]}(\tau, w_1) - r^{[w]}(\tau, w_2) \right\|_{H^1_{\ell}} \le C \left( |\tau| + \max ( \| w_1 \|_{H^1_{p}}, \| w_2 \|_{H^1_{p}} ) \right) \| w_1 - w_2 \|_{H^1_q}, \label{nonlinear-estimate-3} \\
	& \left\| r^{[w]}(\tau_1, w_1) - r^{[w]}(\tau_2, w_2) \right\|_{H^1_{\ell}} \le C \left( |\tau_1 - \tau_2| + \| w_1 - w_2 \|_{H^1_{q}}\right),  \label{nonlinear-estimate-4}\\
	& \left| r^{[\tau]}(\tau_1, w_1) - r^{[\tau]}(\tau_2, w_2) \right| \le C \left( |\tau_1 - \tau_2| + \| w_1 - w_2 \|_{H^1}\right) \label{nonlinear-estimate-5}.
\end{align}
\end{lemma}

\begin{proof}
We denote by $C = C(\delta) > 0$ a generic constant and choose $\delta > 0$ so small such that $B_\delta(0) \subset \U_\tau$ and $B_\delta(0) \subset \U_w$ with $\U_\tau,\U_w$ from Lemma \ref{poly:lemmatrafo} so that the remainders $r^{[f]}, r^{[\tau]}, r^{[w]}$ are well-defined. The main idea for handling the different algebraic rates is to use the following estimates for all $u,v \in H^1_{p} \cap H^1_{q}$ that are simple consequences of standard Sobolev embeddings ($H^1 \subset L^\infty$):
\begin{align} \label{poly:estimatesproof2}
	\big\| |u| |v| \big\|_{L^2_{\ell}}^2 & = \int_\R  |\eta^p(x) u(x)|^2 |\eta^q(x) v(x)|^2 dx \le \| u \|_{L^\infty_p}^2 \| v \|_{L^2_q}^2 \le \| u \|_{H^1_p}^2 \| v \|_{L^2_q}^2
\end{align}
and due to Assumption \ref{A1}, \ref{A2} and \cite[Thm. 2.5]{BeynDoeding22}
\begin{align} \label{poly:estimatesproof1}
\begin{split}
	\big\| |v_\star(\cdot - \tau) - v_\star| |v| \big\|_{L^2_{\ell}}^2 & \le \| v \|_{L^\infty}^2 \int_\R \eta^{2\ell}(x) |v_\star(x - \tau) - v_\star(x)|^2 dx \le C|\tau|^2 \| v \|_{H^1}^2.
\end{split}
\end{align}
In fact, the remainders are quadratic so that we can frequently make use of the estimates \eqref{poly:estimatesproof2}, \eqref{poly:estimatesproof1} in the following. We abbreviate $\kappa(\nu) := v_\star - v_\star(\cdot - \tau) - w_2 + \nu(w_1-w_2)$ and obtain, since $f$ is smooth due to Assumption \ref{A1},
\begin{align*}
	& \| r^{[f]}(\tau,w_1) - r^{[f]}(\tau,w_2) \|_{L^2_{\ell}} \\
	& = \| f(v_\star(\cdot - \tau) + w_1) - f(v_\star(\cdot - \tau) + w_2) - Df(v_\star) (w_1-w_2)\|_{L^2_{\ell}} \\
	& \le \int_0^1 \| [Df(v_\star(\cdot - \tau) + w_2 + \nu(w_1 - w_2)) - Df(v_\star)](w_1-w_2) \|_{L^2_{\ell}} d\nu \\
	& \le \int_0^1 \int_0^1 \| D^2f(v_\star + \sigma \kappa(\nu))[\kappa(\nu),w_1-w_2] \|_{L^2_{\ell}} d\sigma d\nu \\
	& \le \sup \big\{ | D^2 f(z) | : |z| \le 2 \| v_\star \|_{L^\infty} + \| w_1 \|_{L^\infty} + \| w_2 \|_{L^\infty} \big\}  \\
	& \qquad \qquad \qquad \qquad \qquad \qquad \qquad \qquad \qquad \int_0^1 \big\| |\kappa(\nu) | |w_1 - w_2| \big\|_{L^2_\ell} d\nu \\
	& \le C \left( \big\| |v_\star(\cdot - \tau) - v_\star|\,|w_1 - w_2| \big\|_{L^2_{\ell}} + \big\| |w_2|\,|w_1 - w_2| \big\|_{L^2_{\ell}} + \big\| |w_1 - w_2|^2 \big\|_{L^2_{\ell}} \right) \\
	& \le C \left( |\tau| + \max ( \| w_1 \|_{H^1_{p}}, \| w_2 \|_{H^1_{p}} ) \right) \| w_1 - w_2 \|_{H^1_q}.
\end{align*}
The estimate for the derivative follows by similar arguments and by using $f \in C^3$ so that \eqref{nonlinear-estimate-1} holds. For \eqref{nonlinear-estimate-2} we estimate, using the exponential convergence of $v_\star$ at $\pm \infty$ from \cite[Thm. 2.5]{BeynDoeding22},
\begin{align*}
	& \| r^{[f]}(\tau_1 ,w) - r^{[f]}(\tau_2,w) \|_{L^2_{\ell}} = \| f(v_\star(\cdot - \tau_1) +w ) - f(v_\star(\cdot - \tau_2) +w ) \|_{L^2_{\ell}} \\
	& \le \int_0^1 \| Df(v_\star(\cdot - \tau_2) + w + \nu (v_\star(\cdot -\tau_1) - v_\star(\cdot - \tau_2)) \\
	& \qquad \qquad \qquad \qquad \qquad \qquad \qquad (v_\star(\cdot -\tau_1)) - v_\star(\cdot - \tau_2)) \|_{L^2_\ell} d\nu \\
	& \le \sup \big\{ | Df(z) | : |z| \le \| w \|_{L^\infty} + \| v_\star \|_{L^\infty} \big\}  \| v_\star(\cdot -\tau_1)) - v_\star(\cdot - \tau_2)) \|_{L^2_\ell} \\
	& \le C \| v_{\star}(\cdot - \tau_1) - v_{\star}(\cdot - \tau_2) \|_{L^2_{\ell}} \le C |\tau_1 - \tau_2|.
\end{align*}
Again the estimate of the derivative follows analogously and one infers \eqref{nonlinear-estimate-2}. The estimates \eqref{nonlinear-estimate-3}, \eqref{nonlinear-estimate-4}, \eqref{nonlinear-estimate-5} follow from \eqref{nonlinear-estimate-1} and \eqref{nonlinear-estimate-2} and using the boundedness of $P_k: L^2_k \rightarrow \ker(L)$ from Corollary \ref{cor213} and the smoothness of $S_k(\tau)^{-1}$ from \eqref{Sinverse}; see \cite{Doeding} for more details of the estimates.
\end{proof}

\sect{Nonlinear stability}
\label{sec:NonlinearStability}

In this section we prove our main result, Theorem \ref{mainresult}. Assuming suffciently small initial perturbation w.r.t. sufficiently high algebraic rates we prove stability of the solution of the decomposed system \eqref{tau-equation}, \eqref{w-equation} which then lead to the stability of the original system \eqref{perturbsys}. More precisely, we work the corresponding integral equations of \eqref{tau-equation}, \eqref{w-equation}, reading as follows:
\begin{equation} \label{wtau-integral}
\begin{aligned}
	\tau(t) & = \tau_0 + \int_0^t r^{[\tau]}(\tau(s),w(s)) ds, \\
	w(t) & = e^{tL} w_0 + \int_0^t e^{(t-s)L} r^{[w]}(\tau(s),w(s)) ds.
\end{aligned}
\end{equation}
Using the semigroup estimates from Theorem \ref{poly4:semigroup} and the Lipschitz estimates from Lemma \ref{poly:estimates} we show global existence of a solution and algebraic decay of the perturbation $w$ in \eqref{wtau-integral} by applying a Gronwall argument. The transformation \eqref{decomp} then leads to the stability of the original system \eqref{perturbsys}. \\
As a first step we establish the existence of a local (mild) solution of \eqref{wtau-integral} by a standard contraction argument using the exponential estimates of the semigroup from \eqref{semiexpest}. Since the proof is rather standard and fully analogous to the proof of \cite[Lemma 7.1]{BeynDoeding22}, it is omitted here.

\begin{lemma} \label{poly:existence}
  Let the Assumption \ref{A1}-\ref{A4} be satisfied and $k \ge 1$. Further, let $\Cexp \ge 1$ from \eqref{semiexpest} and $\delta > 0$ from Lemma \ref{poly:estimates}. Then for every $0 < \varepsilon_1 < \delta$ and $0 < 2\Cexp \varepsilon_0 \le  \delta$ there
  exists some $t_\star = t_\star(\varepsilon_0,\varepsilon_1) > 0$ such that for all initial values $(\tau_0, w_0) \in \R \times X^1_k$ with $\| w_0 \|_{H^1_k} \le \varepsilon_0$ and $|\tau_0| \le \varepsilon_1$, there exists a unique local (mild) solution $(\tau,w) \in C([0,t_\star),\R \times X^1_k)$ of \eqref{wtau-integral} with 
\begin{align*}
	\| w(t) \|_{H^1_k} \le 2 \Cexp \varepsilon_0, \quad |\tau(t)| \le 2 \varepsilon_1, \quad t \in [0,t_\star).
\end{align*}
In particular, $t_\star$ can be taken uniformly for $(\tau_0,w_0) \in B_{\varepsilon_1}(0) \times B_{\varepsilon_0}(0) \subset \R \times X^1_k$.
\end{lemma}

Next we extend the local solution from Lemma \ref{poly:existence} to a global solution such that the perturbation $w$ decays algebraically. For this we use the following Gronwall estimate involving algebraic integral kernels.

\begin{lemma} \label{poly:Gronwall}
Suppose $p \ge 1$, $C_1 \ge 1$ and $C_2, \varepsilon > 0$ such that
\begin{align*}
	\varepsilon \le  \frac{1}{9C_1C_2C_3}, \quad C_3 := \frac{2^{p-1} p}{p-1}
\end{align*}
and let $\varphi \in C([0,T),\R_+)$ for some $T > 0$ satisfying
\begin{align*}
	\varphi(t) \le \frac{C_1 \varepsilon}{(1+t)^{p}} + C_2 \int_0^t \frac{\varepsilon + \varphi(s)}{(1+t-s)^{p}} \varphi(s) ds \quad \forall t \in [0,T).
\end{align*}
Then for all $0 \le t < T$ there holds
\begin{align*}
	\varphi(t) \le \frac{3C_1 \varepsilon}{(1+t)^{p-1}}.
\end{align*}
\end{lemma}

\begin{proof} We estimate for $t \ge 0$
\begin{align*}
	& \int_0^t \frac{(1+t)^{p-1}}{(1+s)^{p-1}(1+t-s)^{p}} ds = \int_0^1 \frac{t(1+t)^{p-1}}{(1+\tau t)^{p-1}(1+(1-\tau)t)^{p}} d\tau \\
	& \le \int_0^{\frac{1}{2}} \frac{(1+t)^{p}}{(1+\tau t)^{p-1}(1+(1-\tau)t)^{p}} d\tau + \int_{\frac{1}{2}}^1 \frac{t(1+t)^{p-1}}{(1+\tau t)^{p-1}(1+(1-\tau)t)^{p}} d\tau \\
	& \le \int_0^{\frac{1}{2}} \frac{(1+t)^{p}}{(1+\tfrac{1}{2}t)^{p}} d\tau +  \frac{(1+t)^{p-1}}{(1+\tfrac{1}{2} t)^{p-1}} \int_{\frac{1}{2}}^1 \frac{t}{(1+(1-\tau)t)^{p}} d\tau \\
	& = \frac{(1+t)^{p}}{2(1+\tfrac{1}{2}t)^{p}} + \frac{(1+t)^{p-1}}{(p-1)(1+\tfrac{1}{2} t)^{p-1}}\left( 1 - (1+ \tfrac{1}{2} t)^{1-p}) \right) \\
	& \le \frac{(1+t)^{p}}{2(1+\tfrac{1}{2}t)^{p}} + \frac{(1+t)^{p-1}}{(p-1)(1+\tfrac{1}{2} t)^{p-1}} \le 2^{p-1} \Big(1+ \frac{1}{p-1} \Big) = \frac{2^{p-1} p}{p-1} =: C_3
\end{align*}
where we used the bound $(1 + t)^k \le 2^k (1 + \tfrac{1}{2}t)^k$ for all $t,k \ge 0$. Now let
\begin{align*}
	\tau := \sup \big\{ t_\star \in [0,T): \varphi(t) \le 3C_1 \varepsilon (1+t)^{-p+1} \, \forall \, t \in [0,t_\star) \big\}
\end{align*}
and note that $\tau>0$ follows from $\varphi(0)\le C_1 \varepsilon$.
Now assume  $\tau < T$. Then  $\varphi \in C([0,T),\R_+)$ and our assumptions imply
\begin{align*}
	& 3C_1 \varepsilon = (1+\tau)^{p-1} \varphi(\tau) \le C_1 \varepsilon + C_2 \int_0^\tau \frac{(1+\tau)^{p-1}}{(1+\tau-s)^{p}} (\varepsilon + \varphi(s)) \varphi(s) ds \\
	& < C_1 \varepsilon + 3C_1 C_2 \varepsilon \int_0^1 \frac{(1+\tau)^{p-1}}{(1+s)^{p-1}(1+\tau-s)^{p}} (\varepsilon + \varphi(s)) ds \\ 
	& < C_1 \varepsilon + 3C_1C_2 \varepsilon^2  \int_0^\tau \frac{(1+\tau)^{p-1}}{(1+s)^{p-1}(1+\tau-s)^{p}} ds \\
	& \qquad \qquad + 9C_1^2 C_2 \varepsilon^2 \int_0^\tau \frac{(1+\tau)^q}{(1+s)^{2p-2}(1+\tau-s)^{p}} ds \\
	& \le C_1 \varepsilon + 3C_1 C_2 C_3 \varepsilon^2 + 9 C_1^2 C_2 C_3 \varepsilon^2 < 3C_1\varepsilon.
\end{align*}
This is a contradiction,  so that $\tau = T$ follows, and our assertion is proved.
\end{proof}

Now we are in the position to prove the analogue of our main result for the $(\tau,w)$-system \eqref{wtau-integral}.

\begin{theorem} \label{poly:Stabilitywt}
Let Assumption \ref{A1}-\ref{A4} be satisfied and let $k,m,\mu > 0$ be given with $\frac{9}{2} < m < m + \mu \le k$. Then exist constants $\varepsilon > 0$, $K \ge 1$ such that the following statements hold for all initial values $(\tau_0,w_0) \in \R \times X^2_k$ which satisfy $\| (\tau_0, w_0) \|_{\R \times H^1_{k+m+\mu}} < \varepsilon$:
\begin{enumerate}
\item The system \eqref{tau-equation},\eqref{w-equation} has a unique solution 
\begin{align} \label{tauwunique}
	\tau \in C^1([0,\infty),\R), \quad  w \in C((0,\infty),X^2_k) \cap C^1([0,\infty),X_k).
\end{align}
The solution is regular in the sense
\begin{align*}
	w \in C^\alpha((0,\infty),X^2_k) \cap C^{1+\alpha}((0,\infty),X_k) \cap C^1([0,\infty),X_k) \quad \forall \alpha \in (0,1).
\end{align*}
\item There exist $\tau_\infty = \tau_\infty(\tau_0,w_0) \in \R$ such that for all $t \ge 0$
\begin{align}
	& \| w(t) \|_{H^1_k} \le \frac{K}{(1+t)^{\frac{m^*-2}{2}}} \| (\tau_0, w_0) \|_{\R \times H^1_{k+m+\mu}}, \label{w-estimate} \\
	&  |\tau(t) - \tau_\infty| \le \frac{K}{(1+t)^{\frac{m^*-4}{2}}} \| (\tau_0, w_0) \|_{\R \times H^1_{k+m+\mu}}, \label{tau-estimate} \\
	& |\tau_\infty| \le (K+1) \| (\tau_0, w_0) \|_{\R \times H^1_{k+m+\mu}}, \label{tauinfty-estimate}
\end{align}
where $m^* = \lfloor m \rfloor + \max(0,2q-1)$ for $m = \lfloor m \rfloor + q$, $0 \le q < 1$.
\end{enumerate}
\end{theorem}

\begin{proof}
Recall the constants $\Cexp$ from \eqref{semiexpest}, $\Calg$ from \eqref{poly:semi} and $C,\delta > 0$ from Lemma \ref{poly:estimates}. Choose $\varepsilon, \tilde{\varepsilon} > 0$ such that 
\begin{align} \label{poly:epscond}
	2 \Cexp \tilde{\varepsilon} < \delta, \quad  	\varepsilon < \min \left( \frac{\delta}{C_0}, \frac{\tilde{\varepsilon}}{6\Calg}, \frac{2^{-\frac{m-2}{2}} (m-2)}{9m\Calg^2CC_0} \right), \quad C_0 := 2 + \frac{12\Calg C}{m-4}.
\end{align}
We abbreviate $\xi_0:= \| (\tau_0,w_0) \|_{\R \times H^1_{k+m+\mu}} < \varepsilon$ and let
\begin{align*}
	t_\infty : = \sup \Big\{ T>0:\, & \exists ! (\tau,w) \in C([0,T), \R \times X_k) \text{ satisfying \eqref{wtau-integral} on }[0,T), \\
	&  \text{and } \| w(t) \|_{H^1_k} \le \Cexp \tilde{\varepsilon},\, |\tau(t)| \le C_0 \xi_0,\, t \in [0,T) \Big\}.
\end{align*}
Then Lemma \ref{poly:existence} with $\varepsilon_0 = \tilde{\varepsilon}$ and $\varepsilon_1 = \frac{C_0 \xi_0}{2} < \delta$ implies $t_\infty \ge t_\star = t_\star(\varepsilon_0, \varepsilon_1)$. Using Theorem \ref{poly4:semigroup} and Lemma \ref{poly:estimates} (with $q = m + \mu$ and $p = k$) we estimate for all $0 \le t < t_\infty$
\begin{align*}
	& \| w(t) \|_{H^1_k} \le \| e^{tL} w_0 \|_{H^1_k} + \int_0^t \| e^{(t-s)L} \, r^{[w]}(\tau(s),w(s)) \|_{H^1_k} ds \\
	& \le \frac{\Calg}{(1+t)^{\frac{m^*}{2}}} \| w_0 \|_{H^1_{k+m+\mu}} + \int_0^t \frac{\Calg}{(1+t-s)^{\frac{m^*}{2}}} \| r^{[w]}(\tau(s),w(s) \|_{H^1_{k+m+\mu}} ds \\
		& \le \frac{\Calg}{(1+t)^{\frac{m^*}{2}}} \| w_0 \|_{H^1_{k+m+\mu}} + C \Calg \int_0^t \frac{|\tau(s)| + \| w(s) \|_{H^1_{m + \mu}}}{(1+t-s)^{\frac{m^*}{2}}} \| w(s) \|_{H^1_k} ds \\
	& \le \frac{\Calg}{(1+t)^{\frac{m^*}{2}}} \| w_0 \|_{H^1_{k+m+\mu}} + C \Calg \int_0^t \frac{|\tau(s)| + \| w(s) \|_{H^1_{k}}}{(1+t-s)^{\frac{m^*}{2}}} \| w(s) \|_{H^1_k} ds \\
	& \le \frac{\Calg}{(1+t)^{\frac{m^*}{2}}}\xi_0 + C \Calg C_0 \int_0^t \frac{\xi_0 + \| w(s) \|_{H^1_{k}}}{(1+t-s)^{\frac{m^*}{2}}} \| w(s) \|_{H^1_k} ds.
\end{align*}
Thus the Gronwall estimate in Lemma \ref{poly:Gronwall} (with $p = \frac{m^*}{2}$) implies due to \eqref{poly:epscond}
\begin{align} \label{poly:proof1}
	\| w(t) \|_{H^1_k} \le \frac{3\Calg\xi_0 }{(1+t)^{\frac{m^*-2}{2}}} \le \frac{\tilde{\varepsilon}}{2}, \quad \forall \,t \in [0,t_\infty).
\end{align}
This yields, due to $m^* > 4$ and \eqref{poly:epscond}, for all $0 \le t < t_\infty$
\begin{align}
\begin{split} \label{poly:proof2}
	|\tau(t)| & \le |\tau_0| + \int_0^t |r^{[\tau]}(\tau(s),w(s))|ds \le \xi_0 + C \int_0^t \| w(s) \|_{H^1} ds \\
	& \le \xi_0 + C \int_0^t \| w(s) \|_{H^1_k} ds \le \xi_0 + 3C\Calg \xi_0 \int_0^t (1+s)^{-\frac{m^*-2}{2}} ds \\
	& \le \xi_0 + \frac{6C\Calg \xi_0}{m^*-4} = \frac{C_0 \xi_0}{2}.
\end{split}
\end{align}
Next we show $t_\infty = \infty$. Assuming the contrary, $t_\infty < \infty$, then the estimates \eqref{poly:proof1}, \eqref{poly:proof2} imply
\begin{align*}
	\| w(t_\infty - \tfrac{1}{2}t_\star) \|_{H^1_k} \le \frac{\tilde{\varepsilon}}{2} = \varepsilon_0, \quad |\tau(t_\infty - \tfrac{1}{2}t_\star)| \le \frac{C_0 \xi_0}{2} = \varepsilon_1.  
\end{align*}
Setting $\tilde{w}_0 = w(t_\infty - \tfrac{1}{2}t_\star)$ and $\tilde{\tau}_0 = \tau(t_\infty - \tfrac{1}{2} t_\star)$ we can apply Lemma \ref{poly:existence} to the system \eqref{wtau-integral} with initial values $\tilde{\tau}_0, \tilde{w}_0$ and obtain a solution $(\tilde{\tau},\tilde{w})$ on $[0,t_\star)$ such that $\| \tilde{w}(t) \|_{H^1_k} \le \Cexp \tilde{\varepsilon}$ and $|\tau(t)| \le C_0 \xi_0$ for all $t \in [0,t_\star)$. Then the concatenated functions
\begin{align*}
	(\bar{\tau},\bar{w})(t) := \begin{cases} (\tau,w)(t), & t \in [0,t_\infty - \tfrac{1}{2} t_\star ] \\
	(\tilde{\tau}, \tilde{w})(t - t_\infty + \tfrac{1}{2}t_\star), & t \in (t_\infty - \tfrac{1}{2}t_\star,t_\infty + \tfrac{1}{2}t_\star)
	\end{cases}
\end{align*}
define a solution of \eqref{wtau-integral} on $[0, t_\infty + \tfrac{1}{2} t_\star )$ with $\| \bar{w}(t) \|_{H^1_k} \le \Cexp \tilde{\varepsilon}$ and $| \bar{\tau}(t) | \le C_0 \xi_0$ for all $t \in [0,t_\infty + \tfrac{1}{2} t_\star)$. This contradicts the definition of $t_\infty$, so that $t_\infty = \infty$ follows. The estimate \eqref{w-estimate} is a consequence of \eqref{poly:proof1}. For \eqref{tau-estimate} we estimate with Lemma \ref{poly:estimates} and \eqref{poly:proof1} the integral
\begin{align} \label{poly:proof3}
\begin{split}
	& \int_t^\infty |r^{[\tau]}(\tau(s),w(s))| ds \le C \int_t^\infty \| w(s) \|_{H^1_k} ds \\
	& \le 3 C \Calg \xi_0 \int_t^\infty (1 + s)^{- \frac{m^* - 2}{2}} ds = \frac{6C\Calg \xi_0}{m^*-4}(1+t)^{-\frac{m^*-4}{2}}.
\end{split}
\end{align}
Thus, the asymptotic phase
\begin{align*}
	\tau_\infty := \tau_0 + \int_0^\infty r^{[\tau]}(\tau(s), w(s)) ds
\end{align*}
exists and \eqref{tau-estimate} holds due to \eqref{poly:proof3}. The estimate \eqref{tauinfty-estimate} is a simple consequence of \eqref{tau-estimate}. Next we show the regularity of the solution $(\tau,w)$. From Lemma \ref{poly:estimates}
one infers $r^{[\tau]}(\tau(\cdot),w(\cdot)) \in C([0,\infty),\R)$, so that $\tau \in C^1([0,\infty),\R)$. Now let $r(t) := r^{[w]}(\tau(t),w(t))$ for $t \in [0,\infty)$. Then $r$ is Lipschitz continuous since with Lemma \ref{poly:estimates} we have for all $0 \le s < t < \infty$
\begin{align*}
	& \| r(t) - r(s) \|_{L^2_k} \le C | z(t) - z(s)| + C \| w(t) - w(s) \|_{H^1_k} \\
	& \le C \int_s^t | r^{[\tau]}(\tau(\sigma), w(\sigma)) | d\sigma + C \int_s^t \| e^{(t -\sigma)L} r^{[w]}(\tau(\sigma), w(\sigma)) \|_{H^1_k} d\sigma \\
	& \le C \int_s^t | r^{[\tau]}(\tau(\sigma), w(\sigma)) | d\sigma + C\Calg \int_s^t \| r^{[w]}(\tau(\sigma), w(\sigma)) \|_{H^1_{k+m+\mu}} d\sigma \\
	& \le C^2 \int_s^t \| w(\sigma)) \|_{H^1_k} d\sigma + C^2 \Calg \int_s^t |\tau(\sigma)| + \| w(\sigma)) \|_{H^1_{k}} d\sigma \le \tilde{C} (t-s)
\end{align*}
for some $\tilde{C} > 0$. In particular, $r$ is H\"older continuous with index $\alpha \in (0,1)$, i.e., $r \in C^\alpha([0,\infty),X_k)$, and for arbitrary $s > 0$
\begin{align*}
	\int_0^s \| r(t) \|_{L^2_k} dt & = \int_0^s \| r^{[w]}(\tau(t), w(t)) \|_{L^2_k} dt \le C \int_0^s \| w(t) \|_{H^1_k} dt  < \infty.
\end{align*}
The standard theory of parabolic equations, e.g. \cite[Theorem 1.2.1]{Amann} and \cite[Theorem 3.2.2]{Henry}, now implies that $(\tau,w)$ is the unique solution of \eqref{tau-equation}, \eqref{w-equation} in the class \eqref{tauwunique} and that $w \in C^\alpha((0,\infty),X^2_k) \cap C^{1+\alpha}((0,\infty),X_k) \cap C^1([0,\infty),X_k)$ holds for every $\alpha \in (0,1)$. This completes the proof.
\end{proof}

\begin{remark} \label{remark:exponent}
Let us recall the restriction on the values of $m$ and $k$ that determines the required algebraic decay of the initial perturbation $(\tau_0,w_0)$ in Theorem \ref{poly:Stabilitywt}. Roughly speaking, Theorem \ref{poly:semi} guarantees an algebraic decay of the semigroup $\| e^{tL} \|_{X^1_k \rightarrow X^1_{k + m + \mu}} \lesssim t^{-\frac{m^*}{2}}$ for arbitrary $m,k \ge 1$. Then the Gronwall Lemma \ref{poly:Gronwall} gives an algebraic decay of $w$ in $H^1_k$ of one degree less, i.e., $\| w \|_{H^1_k} \lesssim t^{-(\frac{m^*}{2} - 1)}$. This is caused by the semigroup which involves an integration in time and hence costs another degree. The convergence of the shift $\tau$ to the asymptotic phase $\tau_\infty$ needs an additional integration in time, so that $|\tau(t) - \tau_\infty| \lesssim t^{- (\frac{m^*}{2} - 2)}$. This requires $m^* > 4$ and therefore $m > 4 +\frac{1}{2} = \frac{9}{2}$. To ensure the assumptions for the Gronwall argument in the proof of Theorem \ref{poly:Stabilitywt} we need $m + \mu \le k$ which leads to $\frac{9}{2} < m < m + \mu \le k$. Summarizing, the weakest possible condition on the initial data is to make $\| (\tau_0,w_0) \|_{H^1_{9 + \mu}}$ small so that Theorem \ref{poly:Stabilitywt} applies with $m> k > \frac{9}{2}$ and $\mu$ sufficiently small.
\end{remark}

From Theorem \ref{poly:Stabilitywt} we can now conclude our main result Theorem \ref{mainresult}.

\begin{proof}[Proof of Theorem \ref{mainresult}]
Let $\delta > 0$ be such that by Lemma \ref{poly:lemmatrafo} the map $T_k^{-1}$ (resp. $\Pi_k^{-1}$) exists and is diffeomorphic on $B_\delta := \{ u \in L^2_k : \| u \|_{L^2_k} \le \delta \}$ (resp. on $P_k(B_\delta)$), i.e., we may choose $\delta$ so small such that $T_k^{-1}(B_\delta) \subset \U_\tau \times \U_w$ (resp. $\Pi^{-1}(P_k(B_\delta)) \subset U_\tau$). Let
\begin{align*}
	C_1 := \sup_{v \in B_\delta} \frac{|\Pi_k^{-1}(P_k(v))|}{\| v \|_{L^2_k}} 
\end{align*}
and let $C_2 > 0$ such that $\| v_\star(\cdot - \tau_1) - v_\star(\cdot - \tau_2) \|_{H^1_{k + m + \mu}} \le C_2 |\tau_1 - \tau_2|$ for all $\tau_1,\tau_2 \in \Pi_k^{-1}(P_k(B_\delta))$. Next we choose $\varepsilon > 0$ such that every solution $(\tau,w)$ of \eqref{tau-equation}, \eqref{w-equation} from Theorem \ref{poly:Stabilitywt} satisfies $w(t) \in T_k^{-1}(B_\delta)$ and $\tau(t) \in \Pi^{-1}_k(P_k(B_\delta))$ for all $t \in [0,\infty)$. We restrict the size of the initial perturbation to
\begin{align*}
	\varepsilon_0 < \min \Big( \frac{\varepsilon}{1 + C_1 + C_1C_2}, \frac{\delta}{2(2C_1K + C_1 + K)} \Big)
\end{align*}
Set $(\tau_0,w_0) := T_k^{-1}(u_0)$ so that we have
\begin{alignat*}{2}
	& \tau_0 = \Pi_k^{-1}(P_ku_0), \quad && |\tau_0| \le C_1 \| u_0 \|_{L^2_{k}}, \\
	& w_0 = u_0 + v_\star - v_\star(\cdot - \tau_0), \quad && \| w_0 \|_{H^1_{k + m + \mu}} \le (1 + C_1C_2) \| u_0 \|_{H^1_{k+m+\mu}}.
\end{alignat*}
This implies $\| (\tau_0, w_0) \|_{H^1_{k+m+\mu}} \le (C_1 + 1 + C_1C_2) \varepsilon_0 < \varepsilon$ so that the solution $(\tau,w)$ of \eqref{tau-equation}, \eqref{w-equation} with initial value $(\tau_0,w_0)$ from Theorem \ref{poly:Stabilitywt} and the asymptotic phase $\tau_\infty$ satisfy \eqref{decayestimates}. Then the function
\begin{align*}
	u(t) = v_\star(\cdot - \tau(t)) + w(t)
\end{align*}
belongs to $C([0,\infty),H^2_k) \cap C^1([0,\infty),L^2_k)$ and satisfy \eqref{perturbsys} by the computation in Section \ref{sec:Decomposition}. It remains to show the uniqueness of $u$. We note that for all $t \in [0,\infty)$ we have by Theorem \ref{poly:Stabilitywt}
\begin{align*}
	& \| u(t) - v_\star \|_{L^2_k} \le \| v_\star(\cdot - \tau(t)) - v_\star \|_{L^2_k} + \| w(t) \|_{L^2_k} \\
	& \le C_1 |\tau(t) - \tau_\infty| + C_1 |\tau_\infty| + \| w(t) \|_{L^2_k} \le  (2C_1K + C_1 + K) \varepsilon_0 < \frac{\delta}{2}.
\end{align*}
Let $\tilde{u}$ be another solution of \eqref{perturbsys} on $[0,T)$ for some $T > 0$ and let
\begin{align*}
	t_0 := \sup \{ t \in [0,T): \| \tilde{u} - v_\star \|_{L^2_k} \le \delta \text{ on } [0,t) \}.
\end{align*}
Then there are $(\tilde{\tau}, \tilde{w})$ solving \eqref{tau-equation}, \eqref{w-equation} on $[0,t_0)$ with initial value $(\tau_0,w_0)$ and such that $\tilde{u}(t) = v_\star(\cdot - \tilde{\tau}(t)) + \tilde{w}(t)$ for all $t \in [0,t_0)$. But then $(\tau,w) = (\tilde{\tau},\tilde{w})$ and $u = \tilde{u}$ on $[0,t_0)$ by the uniqueness of $(\tau,w)$. Assuming $t_0 < T$, we have for $t \in [0,t_0)$
\begin{align*}
	\frac{\delta}{2} \ge \| u(t) - v_\star \|_{L^2_k} = \| \tilde{u}(t) - v_\star \|_{L^2_k} \rightarrow \delta \quad \text{as }t \rightarrow t_0. 
\end{align*}
Thus $t_0 = T$ and $u = \tilde{u}$ on $[0,T)$. This finishes the proof.
\end{proof}

\section*{Acknowledgements}
  This paper is an extended version of parts of the second author’s PhD Thesis \cite{Doeding}.
   His present work was funded by the Deutsche Forschungsgemeinschaft (DFG, German Research Foundation) under Germany's Excellence Strategy – EXC-2047/1 – 390685813.
  The work of the first author was funded by the Deutsche Forschungsgemeinschaft
(DFG, German Research Foundation) – SFB 1283/2 2021 – 317210226.

\appendix

\sect{Roughness of exponential trichotomies}
\label{appendixA}
Considering a linear differential operator $\A w = w_x - A w$ on an interval $J \subset \R$ and with $A \in C_b(J, \R^{n,n})$. Let $\S: J \times J \rightarrow \R^{n,n}$ be the solution operator of $\A$, i.e., $w = S(\cdot,y) \xi \in C^1(J,\R^n)$ solves $\A w = w_x - A w$ in $J$ with $w(y) = \xi \in \R^2$. Then the operator $\A$ is said to have an \textbf{ordinary exponential trichotomy} on $J$ with exponents $\alpha < \nu < \beta$ if is there is a constant $K > 0$ and projectors $P_\kappa: J \rightarrow \R^{n,n}$, $\kappa = \st,\c,\un$ of rank $m_\kappa$ such that $P_\st + P_\c + P_\un = I$ and such that for all $x,y \in J$ there hold
\begin{align*}
	& S(x,y) P_\kappa(y) = P_\kappa(x) S(x,y), \quad \kappa = \st,\c,\un, \\
	& |S(x,y) P_\st(y)| \le K e^{\alpha(x-y)}, \quad |S(x,y) P_\c(y)| \le K e^{\nu(x-y)}, \quad x \ge y, \\
	& |S(x,y) P_\un(y)| \le K e^{\beta(x-y)}, \quad |S(x,y) P_\c(y)| \le K e^{\nu(x-y)}, \quad x \le y.
\end{align*}
In the case we call $(K,\alpha,\nu,\beta,P_{\st,\c,\un})$ the \textbf{data of the ordinary exponential trichotomy}. If $m_\c = 0$ we speak of a \textbf{shifted exponential dichotomy} and if, in addition, $\alpha < 0 < \beta$ we speak of an \textbf{exponential dichotomy}. There are several results regarding the preservation of an exponential trichotomy (or dichotomy) under suitable perturbations of the matrix-valued function $A$; see e.g. \cite{Coppel}, \cite{Hale}, \cite{Beyn91}. We prove a new so called Roughness Theorem for ordinary exponential trichotomies.

\begin{theorem}[Roughness theorem] \label{RoughTricho}
Let $\Omega \subset \R^n$ be a bounded domain, $\K \in \{ \R,\C \}$ and let $\A(s) = \partial_x - A(s,\cdot)$, $s \in \Omega$, $A\in C_b(\Omega \times \R_+, \K^{n,n})$, have an ordinary exponential trichotomy on $\R_+$ for every $s \in \Omega$ with data $(K(s), \alpha(s),\nu(s),\beta(s),P_{\st,\c,\un}(s))$, $\sup_{s \in \Omega} K(s) < \infty$ depending continuously/analytically on $s \in \Omega$. Further assume that $P_\c(s,x)$ is of rank $m_\c = 1$ and has the form $P_\c(s,x) = z(s,x) \psi(s,x)^H$ where $\mathcal{A}(s) z(s,\cdot) = 0$ on $\R_+$ and such that there are $C_1,C_2 > 0$ with
\begin{align*}
	C_1\le e^{-\nu(s) x}|z(s,x)| \le C_2 , \quad  C_1 \le e^{\nu(s) x} |\psi(s,x)| \le C_2 \quad \forall\,x \ge 0,\, s \in \Omega.
\end{align*}
Let $B \in C(\R_+,\K^{n,n})$ with $|B(x)| \le C_B e^{- \varepsilon x}$, $x \ge 0$ for some $0 < 2 \varepsilon < \inf_{s \in \Omega}\min (\nu(s)-\alpha(s), \beta(s) - \nu(s))$. Then the perturbed operator $\tilde{\A}(s) = \A(s) - B$ has an ordinary exponential trichotomy on $\R_+$ with data $(\tilde{K}, \tilde{\alpha}(s), \nu(s),\tilde{\beta}(s), \tilde{P}_{\st,\c,\un}(s))$ depending continuous/analytically on $s \in \Omega$ and with $\tilde{K}$ independent of $s \in \Omega$. In particular, $\tilde{\alpha}(s)$ and $\tilde{\beta}(s)$ are given by $\tilde{\alpha}(s) = \alpha(s) + 2\delta K(s)$ and $\tilde{\beta}(s) = \beta(s) - 2\delta K(s)$ where
\begin{align*}
	\delta \le \frac{\varepsilon}{4\max(K_\infty ,K_\infty^2)}, \quad K_\infty = \sup_{s \in \Omega} K(s).
\end{align*}
\end{theorem}

\begin{proof}
We denote by $C > 0$ a generic constant independent on $s \in \Omega$ and $x \ge 0$. It is sufficient to prove the assertion for the interval $J = [\tau,\infty)$ instead of $\R_+$ where $\tau \ge 0$ is so large that the following condition is satisfied:
\begin{align} \label{RoughnessCondition}
	4 \max(K_\infty,K_\infty^2) C_B e^{-\varepsilon \tau} \le \varepsilon.
\end{align}
Consider the shifted operator $\A_\nu(s) = \A(s) - \nu(s) I$ which has an ordinary exponential trichotomy on $J$ with data $(K(s), \alpha(s) - \nu(s), 0, \beta(s) - \nu(s), P_{\st,\c,\un}(s))$. Setting $Q_{\st}(s) := P_\st(s)$, $Q_{\un}(s) := P_\c(s) + P_\un(s)$ and $R_\st(s) := P_\st(s) + P_\c (s)$, $R_{\un}(s) := P_\un(s)$ we obtain that $\A_\nu(s)$ has a shifted exponential dichotomy on $J$ with data $(K(s), \alpha(s)-\nu(s), 0, Q_{\st,\un}(s))$ and a shifted exponential dichotomy on $J$ with data $(K(s), 0, \beta(s) - \nu(s), R_{\st,\un})$. With $\delta := C_B e^{-\varepsilon \tau}$ the roughness theorem for exponential dichotomies \cite[Chap. 4, Prop. 1]{Coppel} implies due to \eqref{RoughnessCondition} that the perturbed operator $\tilde{\A}_\nu(s) - B$ has a shifted exponential dichotomy on $J$ with data
\begin{align*}
	(\tfrac{5}{2}K(s), \alpha(s) - \nu(s) + 2\delta K(s), - 2\delta K(s), \tilde{Q}_{\st,\un}(s))
\end{align*}
and a shifted exponential dichotomy on $J$ with data
\begin{align*}	
	(\tfrac{5}{2}K(s), 2\delta K(s), \beta(s) - \nu(s) - 2\delta K(s), \tilde{R}_{\st,\un}(s)).
\end{align*}
Furthermore, it is shown in \cite[Prop. 2.4]{Beyn91} that the projectors satisfy the estimates
\begin{align} \label{RoughnessProjEst}
	|Q_\kappa(s,x) - \tilde{Q}_\kappa(s,x)|, |P_\kappa(s,x) - \tilde{P}_\kappa(s,x)| \le C e^{- \varepsilon x}, \quad x \in J, \, \kappa = \st,\un.
\end{align}
In addition, the ranks of the projectors are preserved, i.e., $\mathrm{rank} (Q_\kappa(s,x)) = \mathrm{rank} (\tilde{Q}_\kappa(s,x))$, $\mathrm{rank} (R_\kappa(s,x)) = \mathrm{rank} (\tilde{R}_\kappa(s,x))$, $\kappa = \st,\un$ for all $x \in J$. This implies, due to $m_\c (s) = 1$, that $\ran(\tilde{Q}_\st(s,\tau)) \subset \ran(\tilde{R}_\st(s,\tau))$ with codimension equal to $1$ so that there is $\tilde{w}(s,\tau) \in \K^n$ with $\K^n = \ran(\tilde{Q}_\st(s,\tau)) \, \oplus \, \mathrm{span}(\tilde{w}(s,\tau)) \, \oplus \, \ran(\tilde{R}_\un(s,\tau))$. Next, we show that $\tilde{w}(s,\tau)$ can be chosen such that $\tilde{w}(s,x) = \tilde{\S}_\nu(s,x,\tau) \tilde{w}(s,\tau)$ satisfies
\begin{align} \label{RoughnessExpEst}
	|w(s,x) - \tilde{w}(s,x)| \le C e^{-\varepsilon x}
\end{align}
where $w(s,x) = e^{-\nu(s) x} z(s,x)$ and where $\tilde{\S}_\nu(s,\cdot,\cdot)$ is the solution operator of $\tilde{\A}_\nu(s)$. \\
Consider for $\xi \in C_b(J,\K^n)$ the map
\begin{align} \label{RoughnessT}
\begin{split}
	T(s,\xi)(x) = w(s,x) & + \int_{\tau}^x S_{\nu}(s,x,y) P_\st(s,y) B(y) \xi(y) dy \\
	& - \int_{x}^\infty S_{\nu}(s,x,y) (P_\c(s,y) + P_\un(s,y)) B(y) \xi(y) dy.
\end{split}
\end{align}
Due to the assumptions we have $\nu(s) - \alpha(s) \ge 2\varepsilon$ so that we have the estimate
\begin{align} \label{Tbound}
\begin{split}
	|T(s, \xi )(x)| & \le C_2 + K(s) \int_{\tau}^x e^{(\alpha(s) - \nu(s))(x-y)} C_B e^{- \varepsilon y} dy \| \xi \|_{\infty} + K(s) \int_x^\infty C_B e^{-\varepsilon y} dy \| \xi \|_{\infty} \\
	& \le C_2 + \frac{K(s)C_B}{\nu(s) - \alpha(s) - \varepsilon} e^{- \varepsilon x} \| \xi \|_{\infty} + \frac{K(s)C_B}{\varepsilon} e^{- \varepsilon x} \| \xi \|_{\infty} \le C_2 + C_T e^{-\varepsilon x} \| \xi \|_{\infty}
\end{split}
\end{align}
with $C_T := \frac{2}{\varepsilon} K_\infty C_B$ independent on $s$. Similarly, we have for $\xi_1,\xi_2 \in C_b(J,\K^n)$ the estimate $|T(s, \xi_1)(x) - T(s, \xi_2)(x)| \le C_T e^{-\varepsilon \tau} \| \xi_1 - \xi_2 \|_{\infty}$. Since $C_T e^{-\varepsilon \tau} < 1$, due to \eqref{RoughnessCondition}, $T(s,\cdot)$ is a contraction on $C_b(J,\K^n)$ so that there exists $\xi(s,\cdot) \in C_b(J,\K^n)$ such that $\tilde{\xi}(s,x) = w(s,x) + T(s,\xi(s,x))$ for all $x \in J$. Further, $\xi(s,\cdot)$ depends continuously/analytically on $s \in \Omega$, since $T(s,\cdot)$ does. Using the definition of $T$ from \eqref{RoughnessT} one shows that $\xi(s,\cdot)$ solves $\tilde{\A}_\nu(s) \xi(s,\cdot) = 0$ on $J$ and as in \eqref{Tbound}
\begin{align*}
	| \xi(s,x) - w(s,x) | \le C_T e^{- \varepsilon x} \| \xi(s,\cdot) \|_{\infty} \le \frac{C_T}{1 - C_T e^{-\varepsilon \tau}} e^{- \varepsilon x} \| w(s,\cdot) \|_{\infty} \le C e^{-\varepsilon x}
\end{align*}
where we used the a-priori bound $\| \xi(s,x) \|_{\infty} \le \frac{1}{1 - C_T e^{-\varepsilon \tau}} \| w(s,\cdot) \|_{\infty} \le \frac{C_2}{1 - C_T e^{-\varepsilon \tau}}$. Now we have $\xi(s,\tau) \in \ran(\tilde{R}_\st(s,\tau))$ since
\begin{align*}
	|\tilde{R}_\un(s,\tau) \xi(s,\tau)| & = |\tilde{S}_{\nu}(s,\tau,x)\tilde{R}_\un (s,x) \tilde{S}_{\nu}(s,x,\tau) \xi(s,\tau)| \\
	& \le \tfrac{5}{2}{K}(s) e^{(\beta(s) - \nu(s) - 2\delta K(s))(\tau-x)} \tilde{S}_{\nu}(s,x,\tau) \xi(s,\tau)| \\
	& \le \tfrac{5}{2}{K}(s) e^{(\beta(s) - \nu(s) - 2\delta K(s))(\tau-x)} |\xi(s,x)| \rightarrow 0 \quad \text{as } x \rightarrow \infty
\end{align*}
since $\beta(s) - \nu(s) - 2\delta K(s) \ge \varepsilon$. Assuming $\xi(s,\tau) \in \ran (\tilde{Q}_\st(s,\tau))$ we obtain
\begin{align*}
	0 \neq | w(s,\tau)| & = |S_{\nu}(s,\tau,x) w(s,x)| \le K(s) |w(s,x)| \\
	& \le K(s) |w(s,x) - \xi(s,x)| + K(s) | \tilde{S}_{\nu}(s,x,\tau) \tilde{Q}_\st(s,\tau) \xi(s,\tau)| \\
	& \le K(s) C e^{-\varepsilon x} + \tfrac{5}{2}K(s)^2 e^{(\alpha(s) - \nu(s) + 2\delta K(s))(x-\tau)} \|\xi(s,\cdot)\|_\infty \rightarrow 0 \quad \text{as } x \rightarrow \infty
\end{align*}
where we used $\alpha(s) - \nu(s) + 2\delta K(s) \le - \varepsilon$. This is a contradiction and therefore $\xi(s,\tau) \in \mathrm{span}(\tilde{w}(s,\tau))$. This shows, after normalization of $\tilde{w}(s,\tau)$, that $\tilde{w}(s,\cdot) = \xi(s,\cdot)$ on $J$ and \eqref{RoughnessExpEst} holds. Now take $\tilde{\psi}(s,\tau)$ such that 
\begin{align*}
	\tilde{\psi}(s,\tau)^H \tilde{w}(s,\tau) = 1, \quad \tilde{\psi}(s,\tau)^H \tilde{Q}_\st(s,\tau) = \tilde{\psi}(s,\tau)^H \tilde{R}_\un(s,\tau) = 0.
\end{align*}
Setting $\tilde{\psi}(s,x) = \tilde{S}^*_{\nu} (s,x,\tau)\tilde{\psi}(s,\tau)$, where $\tilde{S}^*_{\nu}(s,\cdot,\cdot)$ is the solution operator of the adjoint $\tilde{\mathcal{A}}^*_{\nu}(s)$, we find that $\tilde{P}_\c(s,x) = \tilde{w}(s,x)\tilde{\psi}(s,x)^H$ is the projector onto $\mathrm{span} \{ \tilde{w}(s,x) \}$ satisfying $\K^n = \ran(\tilde{Q}_\st(s,x)) \oplus \ran(\tilde{P}_\c(s,x)) \oplus \ran(\tilde{R}_\un(s,x))$. Moreover, $\tilde{P}_\c(s)$ depends continuously/analytically on $s$, since $\tilde{\psi}(s,\cdot)$ does, and from \eqref{RoughnessProjEst} we deduce the estimate $|P_\c(s,x) - \tilde{P}_\c(s,x)| \le C e^{- \varepsilon x}$ for $x \in J$. By assumption and the estimate \eqref{RoughnessExpEst} we have that $\tilde{w}(s,\cdot)$ is uniformly bounded from below. This gives
\begin{align*}
	| \psi(s,x) - \tilde{\psi}(s,x)| & = |\tilde{w}(s,x)|^{-1}|\tilde{w}(s,x)(\psi(s,x) - \tilde{\psi}(s,x))^H| \le C |\tilde{w}(s,x)(\psi(s,x) - \tilde{\psi}(s,x))^H| \\
	& \le C |P_\c(s,x) - \tilde{P}_\c(s,x)| + C |\tilde{w}(s,x) - w(s,x)||\psi(s,x)| \le C e^{-\varepsilon x}.
\end{align*}
Finally this implies for all $x,y \in J$ 
\begin{align*}
	| \tilde{S}_\nu(s,x,y) \tilde{P}_\c(s,y) | = |\tilde{w}(s,x) \tilde{\psi}(s,y)^H| \le |\tilde{w}(s,x)| |\tilde{\psi}(s,y)| \le C.
\end{align*}
With $\tilde{P}_\st(s) = \tilde{Q}_\st(s)$ and $\tilde{P}_\un(s) = \tilde{R}_\un(s)$ we have shown that the operator $\tilde{A}_\nu(s)$ has an exponential trichotomy on $J$ with data 
\begin{align*}
	(\tilde{K}, \alpha(s) - \nu(s) + 2\delta K(s), 0, \beta(s) - \nu(s) - 2\delta K(s), \tilde{P}_{\st,\c,\un}(s))
\end{align*}
depending continuously/analytically on $s \in \Omega$ and some $\tilde{K} > 0$ independent on $s \in \Omega$. Now the assertion for $\tilde{A}(s)$ follows. 
\end{proof}

\sect{A matrix lemma}
\label{appendixB}
For a matrix $A \in \C^{m,m}$ define its lower spectral bound by $\lambda^-(A) =  \min \{\Re x^{H} A x : x \in \C^m, |x|=1 \}$ and note that $  \lambda^-(A) =  \min \{ x^{\top} A x : x \in \R^m, |x|=1 \}$ if $A \in \R^{m,m}$.
\begin{lemma} \label{appB:block}
  Let $A, B \in \R^{m,m}$ and  $C \in \C^{m,m}$. Then the block matrix
  \begin{align*}
    M= M(A,B,C)= \begin{pmatrix} 0 & I_m \\ A^{-1}C & -A^{-1}B \end{pmatrix}
  \end{align*}
  has the following properties:
  \begin{itemize}
  \item[(i)] If $\lambda^-(A), \lambda^-(C)>0$ and $|B-B^{\top}|^2 < 16 \lambda^-(A)  \lambda^-(C)$ then $M$ is hyperbolic with
    dimensions $m_{\st}=m_{\un}=m$.
  \item[(ii)] If $\lambda^-(A)>0$, $\lambda^-(C) \ge 0$, if $B=bI_m$, $b>0$ and if $0$ is a simple eigenvalue of $C$
   then $M$ satisfies $m_{\st}=m$, $m_{\c}=1$, $m_{\un}=m-1$.
  \end{itemize}
\end{lemma}
\begin{proof}
  Note that $M$ has an eigenvalue $\lambda\in \C$ with eigenvector $(x,y)^{\top}$ if and only if $y=\lambda x$ and $(\lambda^2 A + \lambda B  - C)x= 0$. If $\lambda = i \omega$ for some $\omega \in \R$ then by multiplication with $x^H=(x_1+i x_2)^H=x_1^{\top}-i x_2^{\top}$ and taking the real part we find
  \begin{equation*}
    \begin{aligned}
    0& = - \omega^2(x_1^{\top}A x_1 + x_2^{\top}Ax_2)+ \omega x_2^{\top}(B-B^{\top})x_1 - \Re(x^HCx) \\
     & \le \big(- \omega^2 \lambda^-(A)+\frac{1}{2}|\omega||B-B^{\top}|- \lambda^-(C)\big) |x|^2.
    \end{aligned}
  \end{equation*}
  The condition in assertion (i) guarantees that the prefactor of $|x|^2$ is negative for all $\omega \in \R$, hence $M$ is hyperbolic.
  In case of assertion (ii) the prefactor  $-\omega^2\lambda^-(A)$ vanishes only at $\omega=0$. Let us show that $0$
  is a simple eigenvalue of $M$. Since $0$ is a simple eigenvalue of $C$ we have $\ker(C)= \mathrm{span}(v)$ for some $v \neq 0$
  and $v \notin \ran(C)$. Then one finds that $\ker(M)= \mathrm{span}\big( (v,0)^{\top})$ and that $M \left( \begin{smallmatrix}x \\y
  \end{smallmatrix}\right)= \left(\begin{smallmatrix} v \\ 0 \end{smallmatrix}\right)$ holds if and only if $Cx=bv$. Hence $0$ is also
  a simple eigenvalue of $M$. To determine the dimensions we use a homotopy argument. In case (i) the  homotopy
  $(A,(1-t)B,C),t\in [0,1]$ connects $(A,B,C)$ to $(A,0,C)$, and then $((1-t)A+tI_m,0,(1-t)C+tI_m)$ connects $(A,0,C)$ to $(I_m,0,I_m)$.
  Since the homotopy preserves the conditions no eigenvalue can pass the imaginary axis. The final eigenvalue problem $\lambda^2x=x$ has the eigenvalues $\pm 1$ both of multiplicity $m$ which proves our assertion. In case (ii) we just take the homotopy $((1-t)A+t I_m,bI_m, C)$, $t \in [0,1]$
  and arrive at the eigenvalue problem $(\lambda^2+ \lambda b)x= Cx$. Eigenvalues $\mu \in \C$ of $C$ lead to two eigenvalues
  $\lambda_{\pm}= -\frac{b}{2} \pm \sqrt{\mu + \frac{b^2}{4}}$. In case $\mu \neq 0$  we have $\mathrm{Re}(\mu)>0$ and $\Re \lambda_{\pm}(s)$ are of opposite sign due to $b>0$. In case $\mu=0$
  we obtain eigenvalues $0,-b$. Hence we have $m$ stable, $m-1$ unstable and one central eigenvalue.
  Since no eigenvalue can pass the imaginary axis during continuation and
  zero always stays as a simple eigenvalue, the assertion on the dimensions follows.
\end{proof}

\sect{Proof of Lemma \ref{est}}
\label{appendixC}

\begin{proof}
  Throughout the following we use the inequality $ 2^{-\frac{1}{2}}(1+x) \le \eta(x) \le 1 + x$ for $x\ge 0$.
  
  \ref{item1}. 
  Assertion \eqref{eq3:est1} follows for $k > 0$, $q = 1$ and $\beta \ge 0$ by
\begin{align*}
  \eta^k(x)\int_x^{\infty}\eta^{-(k+1)}(y) e^{\beta(x-y)}dy \le (1+x)^k \int_x^{\infty} 2^{\frac{k+1}{2}}(1+y)^{-(k+1)} dy = k^{-1}2^{\frac{k+1}{2}}.
\end{align*}
For  $k \ge 0$ and $0 \le q < 1$ we have
\begin{align*}
  & \eta^{k}(x)\int_x^\infty \eta^{-(k + q)}(y) e^{\beta(x-y)} dy \le \int_x^\infty \eta^{-q}(y) e^{\beta(x-y)} dy \le 2^{\frac{q}{2}} \int_x^\infty (1 + y)^{-q} e^{\beta(x-y)} dy .
\end{align*}
For $\beta\ge1$ the integral is  bounded by $\tfrac{1}{\beta} \le \beta^{q-1}$. If $\beta<1$ we set
$\tau=\tfrac{1}{\beta}-1$ and find for $\tau \ge x$
\begin{align*}
  \int_x^{\infty}(1+y)^{-q}e^{\beta(x-y)}dy & \le \int_x^{\tau} (1+y)^{-q}dy+ (1+\tau)^{-q}\int_{\tau}^{\infty}e^{\beta(x-y)}dy\\
  & \le \frac{(1+\tau)^{1-q}}{1-q}+ (1+\tau)^{-q}\beta^{-1}= \tfrac{2-q}{1-q}\beta^{q-1}.
\end{align*}
For $\tau \le x$ we can omit the first integral and obtain the same estimate from the second one.

\ref{item3}. Similarly to \ref{item1} we have
\begin{align*}
  & \int_0^x \eta^{-q}(y) e^{\beta(y-x)} dy \le 2^{\frac{q}{2}} \int_0^x (1+y)^{-q} e^{\beta(y-x)} dy.
\end{align*}
For $\beta \ge 1$ the integral is bounded by $\beta^{-1} \le \beta^{q-1}$, while for $\beta<1$ we split again at
$\tau= \tfrac{1}{\beta}-1$ and obtain for $\tau \le x$
\begin{align*}
  \int_0^x(1+y)^{-q}e^{\beta(y-x)}dy & \le \int_0^{\tau}(1+y)^{-q}dy + \int_{\tau}^x(1+\tau)^{-q}e^{\beta(y-x)}dy \le \tfrac{2-q}{1-q}\beta^{q-1}.
\end{align*}
If $\tau \ge x$ we just estimate by the first integral.

\ref{item2}. We prove \eqref{eq3:est3} for $k\ge 1$. By the substitution $\xi=1+x$, $\zeta=1+y$ we obtain
\begin{equation*}
\begin{aligned}
  &\eta^k(x)\int_0^x \eta^{-k}(y)e^{\beta(y-x)} dy  \le 2^{\frac{k}{2}}\int_0^x \Big(\frac{1+x}{1+y}\Big)^k
  e^{\beta( y-x)} dy\\
  &= 2^{\frac{k}{2}} \int_1^{\xi} \Big( \frac{\xi}{\zeta}\Big)^k e^{\beta(\zeta-\xi)} d\zeta=: 2^{\frac{k}{2}} T_k(\xi). 
\end{aligned}
\end{equation*}
We choose $\delta>0$ with $\delta \beta_0^{\frac{k-1}{k}} \le \frac{1}{2}$ and  split  the integral at $\tau = \delta \beta^{\frac{k-1}{k}}\xi$ if $\tau \ge 1$. With
$\gamma= \beta(1 - \delta \beta^{\frac{k-1}{k}})$ and $\max_{y\ge0}y^k e^{-\gamma y}= \left(\frac{k}{e \gamma}\right)^k$ we find for $k>1$ and $\tau\ge 1$
\begin{equation} \label{appc:estT}
\begin{aligned}
  T_k(\xi)
  & \le e^{\beta(\tau-\xi)}\xi^k \int_1^{\tau} \zeta^{-k} d \zeta + \Big(\frac{\xi}{\tau}\Big)^k \int_{\tau}^{\xi} e^{\beta(\zeta-\xi)}d\zeta
    \le e^{\beta(\tau-\xi)}\frac{\xi^k}{k-1} +  \Big(\frac{\xi}{\tau}\Big)^k \frac{1}{\beta} \\
 &  \le \frac{\xi^k}{k-1}e^{-\gamma\xi} + \delta^{-k} \beta^{-k} \le \frac{1}{k-1} \Big( \frac{k}{e \gamma}\Big)^k
  + \delta^{-k} \beta^{-k}.
\end{aligned}
\end{equation}
Since $\frac{\beta}{2} \le \gamma \le \beta$ our assertion follows. If $\tau \le 1$ then \eqref{appc:estT}  still
holds with the first summand omitted and thus $T_k(\xi) \le (\delta \beta)^{-k}$. In case $k=1$ we set $\delta=\frac{1}{2}$, $\tau=\frac{\xi}{2}$, $\gamma=\frac{\beta}{2}$.
As above, if $\tau \le 1$ then we can omit the first summand in \eqref{appc:estT} and obtain $T_1(\xi) \le \frac{2}{\beta}$.
Otherwise, the estimate \eqref{appc:estT} gives $T_1(\xi)\le 2 g\big(\tfrac{\xi}{2}\big) + \tfrac{2}{\beta}$ where $g(z)= z \log(z) e^{-\beta z}$ for $z \ge 1$. It remains to estimate $g$. If $0<\beta z \le 1$ then we have $1 \le z \le \beta^{-1}$ and therefore, $g(z)\le z \log(z) \le \beta^{-1} |\log(\beta)|$.
Otherwise we use $\log(\beta z) \le \beta z $  and find
\begin{align*}
  g(z)& = \beta^{-1}(\beta z)\left[ \log(\beta z) - \log(\beta)\right] e^{-\beta z} \le \beta^{-1}\left[ \sup_{y \ge 1}(y \log(y)e^{-y}) + |\log(\beta)| \sup_{y\ge 1}( y e^{-y}) \right]\\
  & \le \beta^{-1} \left[ \sup_{y\ge 1}(y^2 e^{-y})+ e^{-1} |\log(\beta)| \right] \le \beta^{-1}[4e^{-2} +e^{-1} |\log(\beta)|].
  \end{align*}
\end{proof}

\end{document}